\documentclass[paper=letter, fontsize=10pt,leqno]{scrartcl}
\usepackage[T1]{fontenc}
\usepackage[english]{babel}									
\usepackage[protrusion=true,expansion=true]{microtype}		
\usepackage{amsmath,tabu}										
\usepackage{amsfonts}
\usepackage{amsthm}
\usepackage{latexsym}
\usepackage{placeins}
\usepackage{soul}

\usepackage[pdftex]{graphicx}														
\usepackage{url}
\usepackage{amssymb}
\usepackage{bbm}
\usepackage{MnSymbol}

\usepackage[pdftex]{color}
\usepackage{eucal}
\usepackage{amscd}
\usepackage{calrsfs}

\usepackage{sectsty}												
\allsectionsfont{\centering \normalfont\scshape}	

\usepackage{mathtools}

\usepackage{fancyhdr}
\newtheorem{theorem}{Theorem}

\newtheorem{proposition}{Proposition}
\newtheorem{lemma}[proposition]{Lemma}
\newtheorem{corollary}{Corollary}
\newtheorem{definition}{Definition}

\numberwithin{equation}{section}		 
\numberwithin{proposition}{section}			 
\numberwithin{table}{section}				 
\numberwithin{definition}{section}
\numberwithin{theorem}{section}
\numberwithin{corollary}{section}
\numberwithin{exercise}{section}

\title{
		\vspace{-1in} 	
		\usefont{OT1}{bch}{b}{n}
		\normalfont \normalsize \textsc{} \\ [25pt]
		\huge  Stability of periodic orbits in no-slip billiards
}

\date{}

\author{\normalfont \large 
   C. Cox\footnote{Department of Mathematics, Washington University, Campus Box 1146, St. Louis, MO 63130}, 
\ R. Feres\footnotemark[1],  \ H.-K. Zhang\footnote{Department of Mathematics, University of Massachusetts in Amherst} 
}

 \date{\today}

\begin{document}

\maketitle

\begin{abstract}
\begin{center}
 Abstract \end{center}
{\small  
Rigid bodies collision maps in dimension two,  under a natural set of physical requirements, can be classified into two types: the standard specular reflection map and a second which we call, after Broomhead and Gutkin, {\em no-slip.} This leads to the study of {\em no-slip billiards}\----planar billiard systems in which the moving particle is a disc (with rotationally symmetric mass distribution) whose translational and rotational velocities can both change at each collision with the boundary of the billiard domain.

 In this paper we greatly extend previous results on boundedness of orbits (Broomhead and Gutkin) and linear stability of periodic orbits for a Sinai-type billiard (Wojtkowski)  for  no-slip billiards.  We show among other facts that: (i)  for billiard domains in the plane having piecewise smooth boundary and at least one corner of inner angle less than $\pi$,  no-slip billiard dynamics will always contain elliptic period-$2$ orbits; (ii) polygonal no-slip billiards always admit
  small invariant open sets and thus cannot be ergodic with respect to the canonical invariant billiard measure; (iii)  
  the no-slip version of a Sinai billiard  must contain linearly stable periodic orbits of period $2$ and,  more generally, we provide a curvature threshold at which a commonly occurring period-$2$ orbit shifts from being hyperbolic to being elliptic; (iv) finally, we make a number of observations concerning periodic orbits in  a class of polygonal billiards. 
}
\end{abstract}

\section{Introduction and main results}
No-slip billiard systems have   received so far very little attention   despite   some interesting   features  that  distinguish them from the much more widely studied standard billiards.  These non-standard types of billiards are discrete time systems
  in dimension $4$ (after taking the natural Poincar\'e section of a flow)  representing    a rotating disc with unit kinetic energy that moves freely in a billiard domain with piecewise smooth boundary.
Although not Hamiltonian, these systems are nevertheless time reversible and leave invariant the canonical billiard measure. They also exhibit  dynamical behavior that is in sharp contrast with standard billiards. For example, they very often contain  period-$2$ orbits
having small
elliptic islands.    These  regions exist amid chaos that appears, in numerical experiments, to result from the usual mechanisms of dispersing and focusing.
 
We are aware of only two   articles on this subject prior to our \cite{CF,CFII}: one by Broomhead and Gutkin \cite{gutkin} showing  that no-slip billiard orbits
in an infinite strip are bounded, and another by Wojtkowski, characterizing  linear stability 
for a special type of period-$2$ orbit.  
In this paper we extend their results as will be detailed shortly, and develop the basic theory of no-slip billiards in a more systematic way. In this section we explain the organization of the paper and highlight our main new results.

Section \ref{preliminaries}  gives preliminary information and sets notation and terminology concerning rigid collisions. 
It  specializes the general results    from \cite{CF} (stated in that paper in arbitrary dimension for bodies of general shapes and mass distributions) to discs in the plane with rotationally symmetric mass distributions. The main fact is briefly summarized in Proposition \ref{classification}.  Although the classification into specular and no-slip collisions is  the same as in  \cite{gutkin},  our approach is more differential geometric in style and may have some conceptual advantages.  For example, we
derive  this 
classification  (in \cite{CF}) from  an orthogonal decomposition
of the tangent bundle $TM$ restricted to the boundary $\partial M$
 (orthogonal relative to the kinetic energy Riemannian metric in the system's configuration manifold $M$)  into physically meaningful subbundles. This orthogonal decomposition is explained here only for discs in the plane.

By a  planar {\em no-slip billiard} system we mean a mechanical system in $\mathbb{R}^2$  in which one of the colliding bodies, which may have arbitrary shape,  is fixed in place, whereas the second, moving body is a disc with rotationally symmetric mass distribution;  post-collision velocities (translational and rotational) are determined from pre-collision velocities via the no-slip collision map and between consecutive collisions the bodies undergo free motion. Contrary to the standard case, the moving particle's mass distribution influences the collision properties (via an angle parameter which is denoted $\beta$ throughout the paper). 
The main definitions and notations concerning  no-slip billiard systems, in particular  the notions of {\em reduced phase space}, {\em velocity phase space}, and the {\em product, eigen-} and {\em wavefront} frames,  are introduced in Section
 \ref{noslip billiards definitions}.

 A special feature of no-slip billiards around which much of the present study is based, is the ubiquitous occurrence of 
 period-$2$ trajectories. The general description of these trajectories is given in Section \ref{periodic section}. In Section \ref{sec:differential} we obtain the   differential of the no-slip billiard map and show  the form it takes for period-$2$ trajectories. 
 
For  general collision systems  as considered in \cite{CF}  (satisfying energy and momenta conservation, time reversibility, involving impulsive  forces that can only act at the contact point between colliding bodies), the issue of characterizing smooth invariant measures still needs much further investigation, although it is shown there that the canonical (Liouville) measure is invariant if a certain field of subspaces defined in terms of the collision maps at each collision configuration $q\in \partial M$
is parallel with respect to the kinetic energy metric.  This is the case for planar no-slip billiards, so that the standard billiard measure is still invariant. (Note, however, that the configuration manifold is now $3$-dimensional.) A detailed proof of this fact, in addition to comments on time reversibility are given in  
 Section \ref{measure reversible}.
 
 Dynamics proper begins with Sections \ref{wedge section} and \ref{higher order polygon}, which are concerned with no-slip billiard systems in wedge-shaped regions and polygons. Here we generalize the main result from \cite{gutkin}. In that paper it is shown that orbits of the  no-slip billiard system  in an infinite strip domain are bounded. By extending and refining this fact to wedge regions
we obtain local stability for periodic orbits in  no-slip polygonal billiards. This is Theorem \ref{polygon stability}. We also give in Section \ref{higher order polygon} an exhaustive description of periodic orbits in wedge billiards. 

Finally, in Section \ref{curved section} we consider linear stability of period-$2$ orbits in the presence of curvature. Our results here extend those of \cite{W} for no-slip Sinai billiards.
Wojtkowski  makes in \cite{W} the following striking observation: for a special period-$2$ orbit in a Sinai billiard (corresponding in our study to angle $\phi=0$)
there is a parameter defined in terms of  the curvature of the circular scatterer that sets a threshold between elliptic and hyperbolic behavior.  
This is based on an analysis of the differential of the billiard map at the periodic trajectory. Here we derive similar results for period-$2$ orbits in general.
Although the analysis is purely linear, we observe the occurrence of elliptic islands and stable behavior in systems that are the no-slip counterpart of fully chaotic 
standard billiards. The fast transition between stability and chaos near the threshold set by the curvature parameter obtained from the linear analysis is also very apparent numerically. As already observed by Wojtkowski in \cite{W}, proving local stability in the presence of curvature would require a difficult  KAM analysis (for reversible, but not Hamiltonian systems;  cf. \cite{sevryuk}).

\section{Background on rigid collisions}\label{preliminaries}
For a much more general treatment of the material of this section (not restricted to discs and valid in arbitrary dimension) see \cite{CF}. See also \cite{gutkin,W}.
Let $x=(x_1, x_2)$ be the standard coordinates in $\mathbb{R}^2$. A mass distribution on a body $B\subset \mathbb{R}^2$ is defined by
a finite positive measure 
 $\mu_{\text{\tiny mass}}$  on  $B\subset \mathbb{R}^2$. We assume without loss of generality  that the first moments $\int_B x_i d\mu_{\text{\tiny mass}}(x)$ vanish. Let $m$ be the total mass: $m=\mu_{\text{\tiny mass}}(B)$.
The second moments of $\mu_{\text{\tiny mass}}$ (divided by $m$) will be denoted by $\ell_{rs}=\frac1m \int_B x_rx_s\, d\mu_{\text{\tiny mass}}(x)$.  When $B$ is a disc of radius $R$ centered at the origin of $\mathbb{R}^2$, it will be assumed that $\mu_{\text{\tiny mass}}$ is rotationally symmetric, in which case the symmetric matrix $L=(\ell_{rs})$ is scalar: $L=\lambda I$.  Also in this case, 
$0\leq \lambda\leq R^2/2$, where $0$ corresponds to all mass being concentrated at the origin and the upper bound corresponds to having all mass concentrated on the circle of radius $R$. For the  uniform mass distribution, $\lambda= R^2/4$. It will be useful to introduce the parameter $\gamma:=\sqrt{2\lambda}/R$.
The moment of inertia of a disc of radius $R$  is given in terms of $\gamma$ by $\mathcal{I}= m (\gamma R)^2$. 
For the uniform mass distribution on the ball
$\gamma=1/\sqrt{2}$. In general, $0\leq \gamma \leq 1$.

From now on $B$ will be the disc of radius $R$ in $\mathbb{R}^2$  centered at the origin.  A {\em configuration} of the moving disc  is an image of $B$
under a Euclidean isometry; it is parametrized by the coordinates of the center of mass $(y,z)$ of the image disc and its angle of orientation $\theta$. Introducing the new coordinate $x:= \gamma R \theta$ and denoting  by $v=(\dot{x},\dot{y}, \dot{z})$ the velocity vector in configuration space, the kinetic energy of the body takes the form $\frac1{2} m|v|^2$, where $|v|$ is the standard Euclidean norm.

We consider now a system of two discs with radii $R_1, R_2$,   denoted by $B_1, B_2$ when  in their reference configuration, that is, 
when centered at the origin of $\mathbb{R}^2$.
 The configuration manifold $M$  of the system, which is the set of all non-overlapping images of $B_i$ under Euclidean isometries of the plane, 
 is then the set
 $$M:=\left\{ (a_1, a_2)\in (\mathbb{R}^2\times \mathbb{T}_1)\times (\mathbb{R}^2\times \mathbb{T}_2): |\overline{a}_1-\overline{a}_2|\geq R_1+R_2 \right\},$$
 where
 $\mathbb{T}_i =  \mathbb{R}/(2\pi\gamma_i R_i)$ and 
  $\overline{a}$ indicates the coordinate  projection of $a$ in $\mathbb{R}^2$. 
 Given mass distributions $\mu_i$, $i=1,2$, the kinetic energy
 of the system at a state $(a_1, v_1, a_2, v_2)$ is $K=\frac12\left(m_1|v_1|^2+m_2|v_2|^2\right).$ The  manifold $M$ becomes a Riemannian manifold with boundary
 by endowing it with the {\em kinetic energy metric} $$\left\langle (u_1,u_2), (v_1, v_2)\right\rangle=m_1 u_1\cdot v_1 + m_2 u_2\cdot v_2$$
where the dot means ordinary inner product in $\mathbb{R}^3$.

\begin{figure}[htbp]
\begin{center}
\includegraphics[width=2 in]{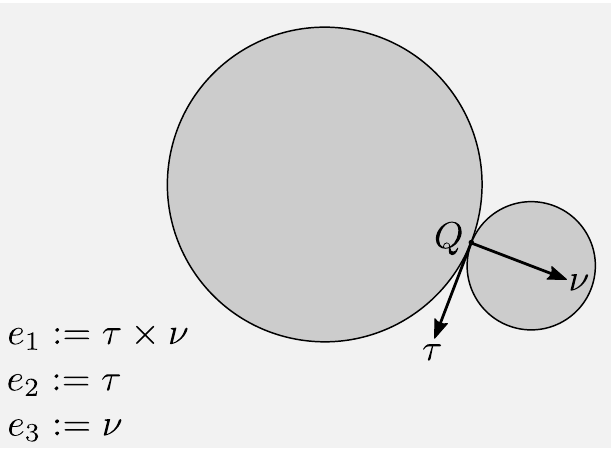}\ \ 
\caption{\small{Definition of the $(e_1, e_2, e_3)$ frame at the contact point $Q$. The unit normal vector $\nu$ points away from the center of (the image in the given configuration of) body $B_1$.}} 
\label{frame_definition}
\end{center}
\end{figure} 

Let $\tau$ and $\nu$ denote the unit tangent and normal vectors 
at the contact point  $Q\in \mathbb{R}^2$ of the bodies in  a  boundary configuration $q\in \partial M$ as indicated in Figure \ref{frame_definition}.  We introduce the orthonormal frame $e_1, e_2, e_3$ of $\mathbb{R}^3$ as defined  in the figure. If the system is at a state $(a_1, v_1, a_2, v_2)$, then $Q$ will have velocity 
$V_i(Q)$  when regarded as a point in body $B_i$ in the given configuration $a_i$. One easily obtains
$$V_1(Q)=\left[v_1\cdot e_2 - \gamma_1^{-1}v_1\cdot e_1\right] e_2 + v_1\cdot e_3 e_3, \ \ V_2(Q)=\left[v_2\cdot e_2 - \gamma_2^{-1}v_2\cdot e_1\right] e_2 + v_2\cdot e_3 e_3. $$
The unit normal vector to $\partial M$  at the configuration $q=(a_1, a_2)$ pointing towards the interior of $M$ will be denoted $\mathbbm{n}_q$. Note that
this is defined with respect to the kinetic energy metric.  Explicitly, letting $m=m_1+m_2$,
$$\mathbbm{n}_q= \left(-\sqrt{\frac{m_2}{m_1 m}}e_3, \sqrt{\frac{m_1}{m_2 m}}e_3 \right). $$
We also define the {\em no-slip}  subspace $\mathfrak{S}_q$  of the tangent space to $\partial M$ at $q$:
$$ \mathfrak{S}_q=\left\{ (v_1, v_2): V_1(Q)=V_2(Q)\right\},$$
and the orthogonal complement to $\mathfrak{S}_q$ in $T_q\partial M$, which we write as $\overline{\mathfrak{C}}_q$. The orthogonal direct sum of the latter and the line spanned
by $\mathbbm{n}_q$ was denoted by $\mathfrak{C}_q$ in \cite{CF} and called the {\em impulse} subspace. These two spaces are given by
\begin{equation}\label{impulse_space}\overline{\mathfrak{C}}_q=\left\{\left(-\frac{1}{\gamma_1}  v_1\cdot e_2,   v_1\cdot e_2, 0; \frac1{\gamma_2} v_2\cdot e_2, v_2\cdot e_2,0\right): m_1 v_1 + m_2 v_2 =0 \right\} \end{equation}
and 
\begin{equation}\label{noslip_space}\mathfrak{S}_q=\left\{(v_1,v_2): (v_1-v_2)\cdot e_3=0, (v_1-v_2)\cdot e_2= \left(\frac{v_1}{\gamma_1}+\frac{v_2}{\gamma_2}\right)\cdot e_1\right\}. \end{equation}
Figure  \ref{orthogonal} shows typical vectors in these orthogonal subspaces. 

\begin{figure}[htbp]
\begin{center}
\includegraphics[width=5 in]{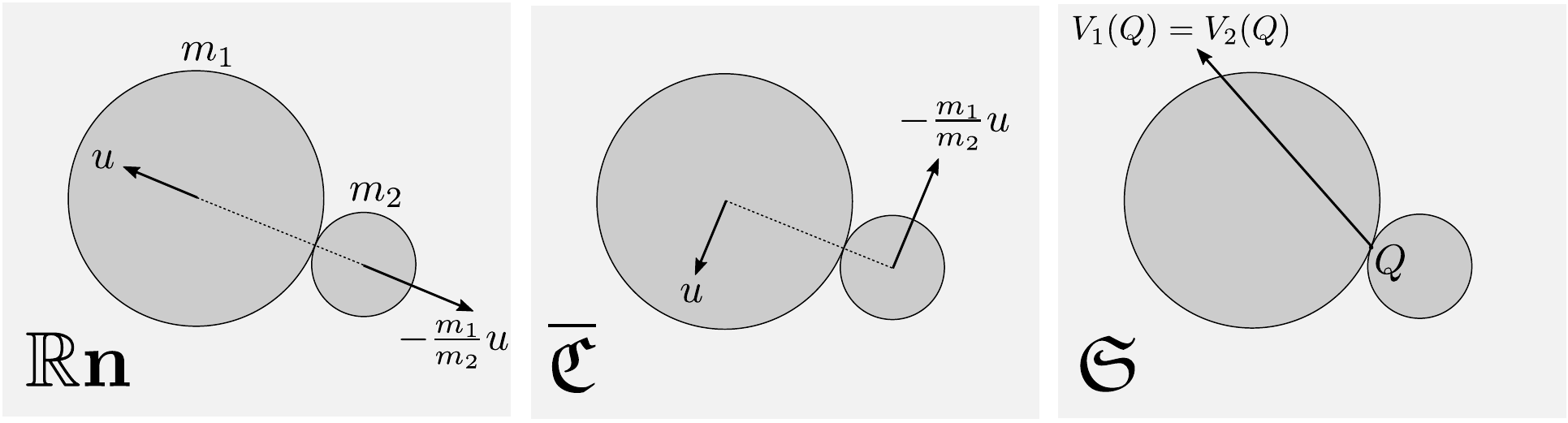}\ \ 
\caption{\small{Tangent vectors in $T_qM$ decompose orthogonally (relative to the kinetic energy metric) into vectors of the above three types. In the diagram on the left the bodies are not rotating, and in the middle diagram the rotation velocities are determined by the  center of mass velocities as given by (\ref{impulse_space}). }} 
\label{orthogonal}
\end{center}
\end{figure} 

We can now 
 state in the present very special case one of the main results of \cite{CF}. See also \cite{gutkin}.
 In the interior of $M$, the motion of the system, in the absence of potential forces, is geodesic relative to the kinetic energy metric. In dimension two this
 amounts simply to constant linear and angular velocities. 
  For the motion   to be fully specified it is necessary to
 find  for any given  $v^-=(v^-_1, v^-_2)$ at $q\in \partial M$ such that $\langle v^-, \mathbbm{n}_q\rangle<0$ (representing the bodies' velocities immediately before 
 a collision) a $v^+$ such that $\langle v^+, \mathbbm{n}_q\rangle>0$ (representing the bodies' velocities immediately after). This correspondence should be given by a map
 $C_q: v^-\mapsto v^+$.  We call such correspondence a {\em collision map} $C_q$ at  $q\in \partial M$. 
 \begin{proposition}\label{classification}
 Linear collision maps $C_q:T_qM\rightarrow T_qM$ at $q\in \partial M$ describing energy preserving, (linear and angular) momentum preserving, time reversible collisions having the additional property that
 impulsive forces between the bodies can only act at  the single point of contact (denoted by $Q$ above)  are given  by  the linear orthogonal involutions  that restrict to the identity map on $\mathfrak{S}_q$ and send $\mathbbm{n}_q$ to its negative. Thus $C_q$ is fully determined by its restriction to $\overline{\mathfrak{C}}_q$, where
 it can only be (in dimension $2$) the identity or its negative.
 \end{proposition}
 We refer to \cite{CF} for a more detailed explanation of this result. The key point for our present purposes is that, in addition to the standard reflection map given
 by specular reflection with respect to the kinetic energy metric, there is only one (under the stated assumptions) other map, which we refer to as the {\em no-slip collision map.}

The explicit form of the no-slip collision map $C_q$ can thus be obtained by first decomposing the pre-collision vector $v^-$ according to the   orthogonal decomposition
$ \mathfrak{S}_q\oplus \mathfrak{C}_q$ and then switching the sign of the $\mathfrak{C}_q$ component to obtain $v^+$. If $\Pi_q$ is the orthogonal projection to $\mathfrak{C}_q$, then
$C_q=I-2\Pi_q.$ Setting $[v]=[(v_1,v_2)]:=\left(v_1\cdot e_1, v_1\cdot e_2, v_1\cdot e_3, v_2\cdot e_1, v_2\cdot e_2, v_2\cdot e_3\right)^t,$ the pre- and post-collision velocities are related by $[v^+] = [C_q] [v^-]$ where 

$$
\begingroup
\renewcommand*{\arraystretch}{1.5}
[C_q]=\left(\begin{array}{cccccc}1-\frac{\delta}{m_1\gamma_1^2} & \frac{\delta}{m_1\gamma_1} & 0 & -\frac{\delta}{m_1\gamma_1\gamma_2} & -\frac{\delta}{m_1\gamma_1} & 0 \\\frac{\delta}{m_1\gamma_1} & 1-\frac{\delta}{m_1} & 0 & \frac{\delta}{m_1\gamma_2} & \frac{\delta}{m_1} & 0 \\0 & 0 & 1-\frac{2 m_2}{m} & 0 & 0 & \frac{2m_2}{m} \\-\frac{\delta}{m_2\gamma_1\gamma_2} & \frac{\delta}{m_2\gamma_2} & 0 & 1-\frac{\delta}{m_2\gamma_2} & -\frac{\delta}{m_2\gamma_2} & 0 \\-\frac{\delta}{m_2\gamma_1} & \frac{\delta}{m_2} & 0 & -\frac{\delta}{m_2\gamma_2} & 1-\frac{\delta}{m_2} & 0 \\0 & 0 & \frac{2m_1}{m} & 0 & 0 & 1-\frac{2 m_1}{m}\end{array}\right).
\endgroup
$$
Here  $m=m_1+m_2$ and $\delta= 2\left\{ \frac1{m_1}\left[1+\frac1{\gamma_1^2}\right]+ \frac1{m_2}\left[1+\frac1{\gamma_2^2}\right]\right\}^{-1}.$ We record  two special cases. First suppose that the two discs have the same mass distribution and    $\gamma=\gamma_i$. Then
$$
\begingroup
\renewcommand*{\arraystretch}{1.5}
[C_q]= \left(\begin{array}{cccccc}\frac{\gamma^2}{1+\gamma^2} & \frac{\gamma}{1+\gamma^2} & 0 & -\frac{1}{1+\gamma^2} & -\frac{\gamma}{1+\gamma^2} & 0 \\\frac{\gamma}{1+\gamma^2} & \frac{1}{1+\gamma^2} & 0 & \frac{\gamma}{1+\gamma^2} & \frac{\gamma^2}{1+\gamma^2} & 0 \\0 & 0 & 0 & 0 & 0 & 1 \\-\frac{1}{1+\gamma^2} & \frac{\gamma}{1+\gamma^2} & 0 & \frac{\gamma^2}{1+\gamma^2} & -\frac{\gamma}{1+\gamma^2} & 0 \\-\frac{\gamma}{1+\gamma^2} & \frac{\gamma^2}{1+\gamma^2} & 0 & -\frac{\gamma}{1+\gamma^2} & \frac{1}{1+\gamma^2} & 0 \\0 & 0 & 1 & 0 & 0 & 0\end{array}\right)
\endgroup
$$

 The second case of interest assumes that one body, say $B_1$ is fixed in place. 
It makes sense to pass to the  limit in which the mass and moment of inertia of $B_1$ approach infinity and its   velocity (which  does not  change during collision process
as a quick inspection of $C_q$ shows) is set equal  to $0$. In this case only the velocity of
$B_2$ changes and we write $m = m_2, v=v_2, \gamma=\gamma_2$ and $v=v_2$. Then
$$\mathbbm{n}_q=\left(0, 0, m^{-1/2}\right), \ \ \mathfrak{S}_q= \left\{(-\gamma s,  s, 0): s\in \mathbb{R}\right\}, \ \ \overline{\mathfrak{C}}_q= \left\{(s, \gamma s, 0): s\in \mathbb{R}\right\} $$
where vectors are expressed in the frame $(e_1, e_2, e_3)$.  
The (lower right block of the) matrix $[C_q]$ is now
\begin{equation}\label{collision_map}\begingroup
\renewcommand*{\arraystretch}{1.5}
[C_q]= \left(\begin{array}{ccr}-\frac{1-\gamma^2}{1+\gamma^2} & -\frac{2\gamma}{1+\gamma^2} & 0 \\-\frac{2\gamma}{1+\gamma^2} & \ \, \, \frac{1-\gamma^2}{1+\gamma^2} & 0 \\0 & 0 & -1\end{array}\right)=\left(\begin{array}{ccc}-\cos\beta & -\sin\beta & 0 \\-\sin\beta &\ \, \cos\beta & 0 \\0 & 0 & -1\end{array}\right).
\endgroup \end{equation}
As noted earlier, $\gamma=\tan^2(\beta/2)$ can take any value between $0$ and $1$ (equivalently, $0\leq \beta\leq \pi/2$) in dimension $2$;   thus it makes sense to define the angle $\beta$ as we did above.
When the mass distribution is uniform, $\gamma=1/\sqrt{2}$ and
$\cos\beta=1/3$, $\sin\beta=2\sqrt{2}/3$.
Notice that these expressions still hold regardless of the shape of the fixed body $B_1$.   In what follows we denote the above $3$-by-$3$  matrix by $\mathcal{C}:=[C_q]$.

\section{No-slip planar billiards}\label{noslip billiards definitions}
We focus attention on the  last case indicated at the end of Section \ref{preliminaries}, which we call a planar no-slip {\em billiard}: the billiard table is the complement of the fixed body $B_1$,   now any set in $\mathbb{R}^2$ with non-empty interior and piecewise  smooth boundary,  and the billiard ball is the disc $B_2$ of radius $R$ with a rotationally symmetric mass distribution. The configuration manifold is then the set  $M$ consisting of all $q=(\overline{q},x)\in \mathbb{R}^2\times \mathbb{T}$  for which the distance between $\overline{q}$ and 
$B_1$ is at least $R$. We define $\mathbb{T}:=\mathbb{R}/(2\pi\gamma R)$. The projection $q\mapsto \overline{q}$ maps $M$ onto the {\em billiard table}, denoted $\mathcal{B}$. The boundary of $M$ is $\partial \mathcal{B}\times \mathbb{T}$ and the frame $(e_1(q), e_2(q), e_3(q))$  at $q\in \partial M$ is as indicated in Figure \ref{frame}.
We also view this $q$-dependent frame as the orthogonal map $\sigma_q:\mathbb{R}^3\rightarrow T_q\mathbb{R}^3$ that sends the standard basis vectors $\epsilon_i$ of $\mathbb{R}^3$ to $e_i(q)$, for $i=1, 2, 3$. This allows us to write $C_q=\sigma_q\mathcal{C}\sigma^{-1}_q $ for each $q$. 

Due to energy conservation, the norms of velocity vectors are not changed during collision or during the free motion between collisions; we restrict attention to vectors of unit length (in the kinetic energy metric, which agrees with the standard Euclidean metric in $\mathbb{R}^3$ under the choice of coordinate $x= \gamma R \theta$ for the disc's angle of orientation).

\begin{figure}[htbp]
\begin{center}
\includegraphics[width=3.0 in]{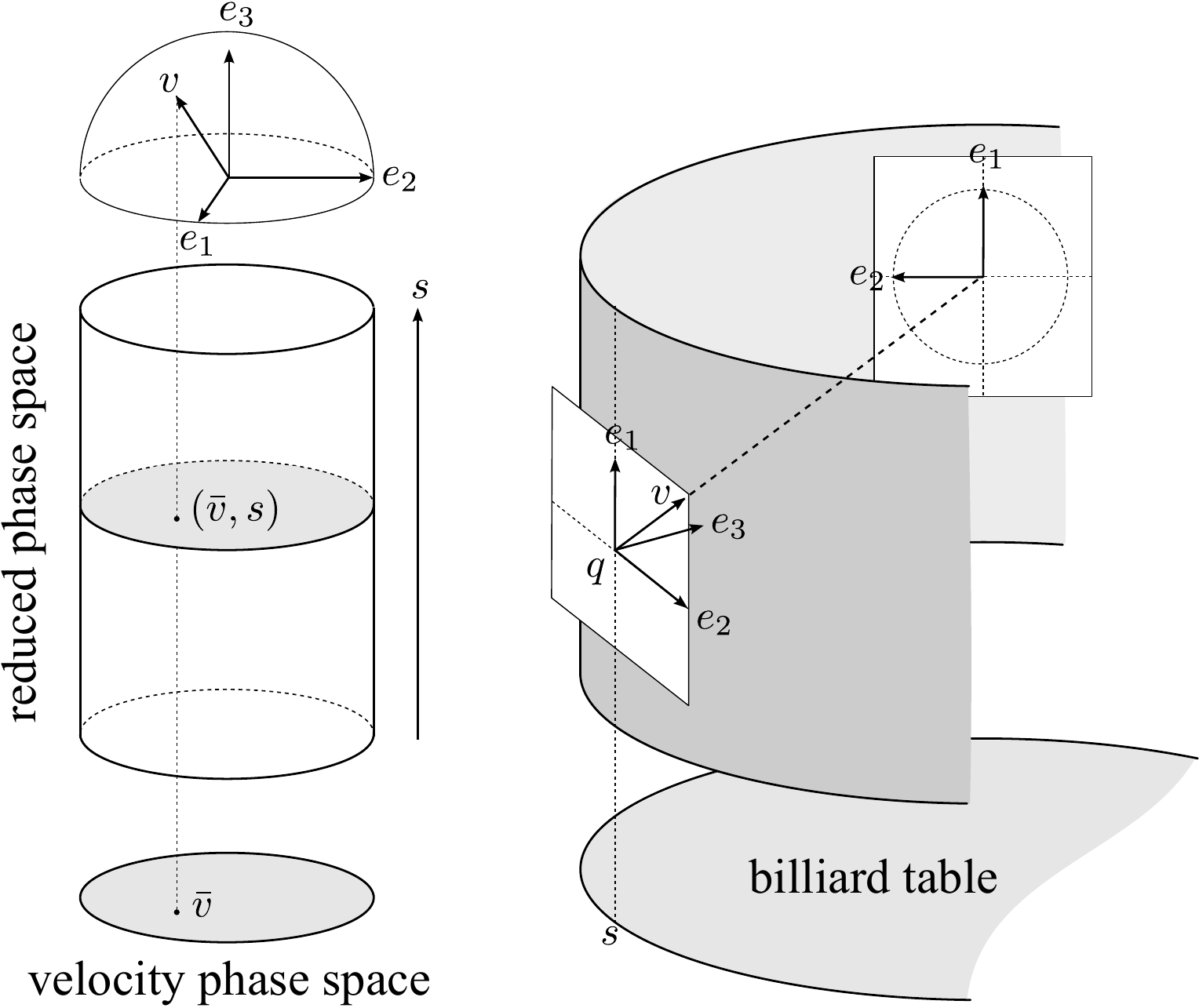}\ \ 
\caption{\small{Definition of the product frame $(e_1, e_2, e_3)$, the reduced phase space, and the velocity phase space.}} 
\label{frame}
\end{center}
\end{figure}

 Throughout the paper the  notations $\langle u,v\rangle$ and $u\cdot v$  are both used for the standard inner product in $\mathbb{R}^3$, the choice   
 being a matter of typographical convenience.  (Recall that the kinetic energy metric has been reduced to the standard inner product in $\mathbb{R}^3$ by our definition of the variable $x$.)
The {\em phase space} of the billiard system is $$N:=N^+:=\{(q,v)\in T\mathbb{R}^3: q\in \partial M, |v|=1, v\cdot e_3(q)>0\}.$$ 
We refer to vectors in $N^+$ as {\em post-collision} velocities; we similarly define the space $N^-$ of {\em pre-collision} velocities. 
The {\em billiard map} $T$, whose domain is a subset of  $N$, is the composition of the free motion between two points $q_1, q_2$ in  $\partial M$ and the billiard map $C_{q_2}$ at the endpoint.
Thus   $T:N\rightarrow N$ is given by
$$ (\tilde{q},\tilde{v})=T(q,v)=(q+tv, C_{\tilde{q}}v)$$
where $t:=\inf \{s>0: q+sv\in N\}$. (We assume that the shape of the billiard table $\mathcal{B}$ is such that $T$ makes sense 
and is smooth for all $\xi$ in some big subset of $N$, say open of full Lebesgue measure.)
 Let 
$$\xi=(q,v)\mapsto \tilde{\xi}_-=(\tilde{q},v)\mapsto \tilde{\xi}=\tilde{\xi}_+=(\tilde{q}, C_{\tilde{q}}v). $$
The first map in this composition, which we denote by $\Phi$, is parallel translation of $v$ from $q$ to $\tilde{q}$,  and the second map, denoted $C$, applies the no-slip reflection map to  the translated vector, still denoted $v$, at $\tilde{q}$.  Hence $T=C\circ \Phi$.

\begin{figure}[htbp]
\begin{center}
\includegraphics[width=4.5 in]{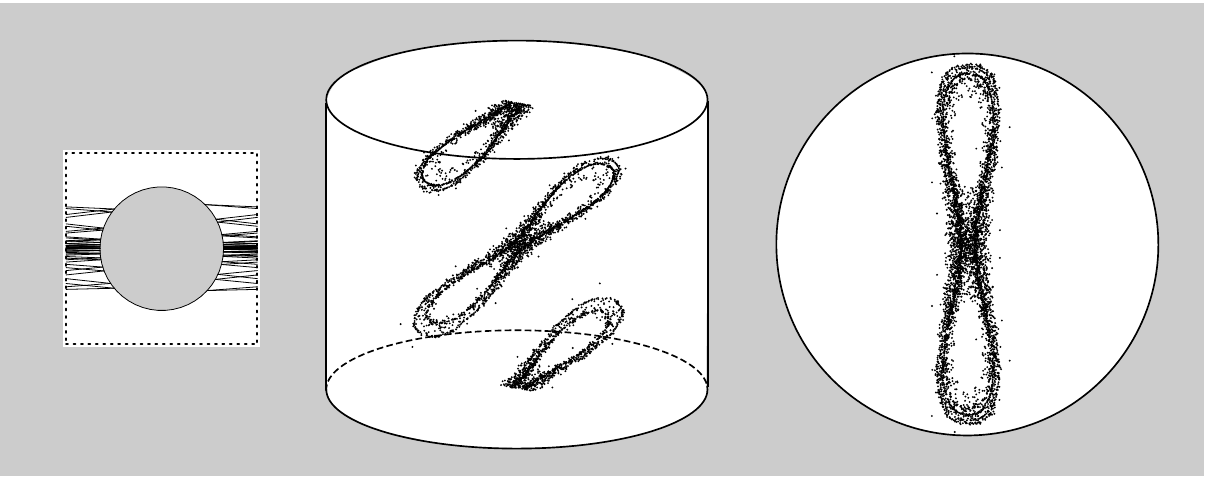}\ \ 
\caption{\small{Left: projection from $\mathbb{R}^3$  to $\mathbb{R}^2$ of an orbit of the no-slip Sinai billiard, to be discussed in more detail later on. Middle: the same orbit shown in the reduced phase space and, on the right, in the velocity phase space.}} 
\label{example}
\end{center}
\end{figure} 

 Taking into account the rotation symmetry of the moving disc, we may for most purposes
ignore the angular coordinate (but not the angular velocity!) and 
 restrict attention to the {\em reduced phase space}. This is defined as $\partial\mathcal{B}\times \{u\in \mathbb{R}^2:|u|<1\}$, where an element $u$ of the unit disc represents the velocity vector at $q\in \partial \mathcal{B}$ (pointing into the billiard region) given by $$\sigma_q \left(u_1, u_2, \sqrt{1-|u|^2}\right)=u_1 e_1(q)+u_2 e_2(q)+\sqrt{1-|u|^2} e_3(q).$$ By {\em velocity phase space} we mean this unit disc. Figure \ref{frame} summarizes these definitions and Figure \ref{example} shows what an orbit segment  
looks like in each space. On the left of the latter figure is shown the two dimensional projection of the orbit segment defined in the $3$-dimensional space $M$. 

The rotation symmetry that justifies passing from the $4$-dimensional phase space  to the $3$-dimensional  reduced phase space may be formally expressed by the identity   $$T(q+\lambda e_1, v)= T(q,v)+ \lambda e_1.$$ Keep in mind that  $e_1$ is a parallel vector field (independent of $q$). In particular, $e_1$ is invariant under the differential map: $dT_\xi e_1 = e_1$ for all  $\xi=(q,v)$.

In addition to the orthonormal frames $\sigma_q$ it will be useful to introduce a frame consisting of eigenvectors of the collision map $C_q$.  We define
\begin{align}\label{u in e}
\begin{split}
u_1(q) &= \sin(\beta/2) e_1(q) - \cos(\beta/2) e_2(q)\\
u_2(q) & = \cos(\beta/2) e_1(q) + \sin(\beta/2) e_2(q)\\
u_3(q) &= e_3(q).
\end{split}
\end{align}
See Figure \ref{eigenframe}.
Then $$C_q u_1(q) = u_1(q), \ \ C_q u_2(q)=-u_2(q), \ \ C_q u_3(q)= -u_3(q).$$

Yet a third orthonormal frame will prove useful later on in our analysis of period-$2$ trajectories.
Let $\xi=(q,v)\in N$. Then $w_1(\xi), w_2(\xi), w_3(\xi)$ is the orthonormal frame at $q$ such that  
$$w_1(\xi):=\frac{e_1(q) - e_1(q)\cdot v v}{|e_1(q) - e_1(q)\cdot v v|}, \ \ w_2(\xi):=v\times w_1(\xi),\ \ w_3(\xi):=v.$$
Note that $w_1(\xi)$ and $w_2(\xi)$ span the $2$-space perpendicular to $v$. 

\begin{figure}[htbp]
\begin{center}
\includegraphics[width=4.5 in]{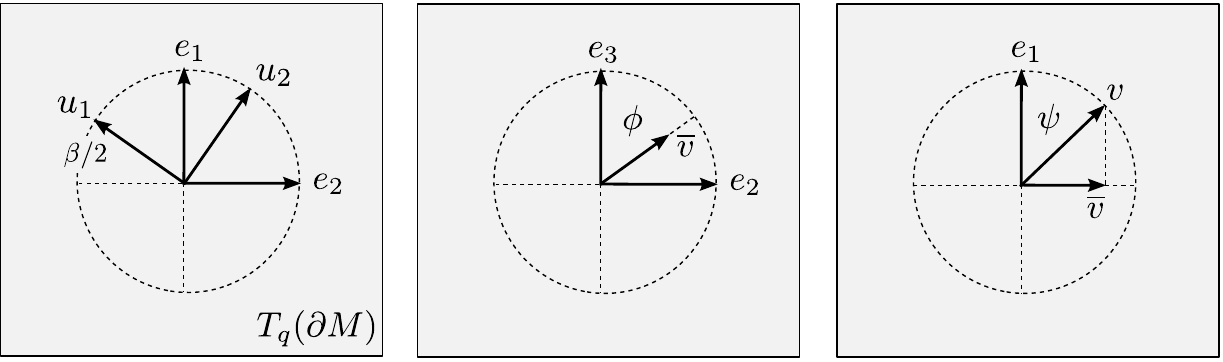}\ \ 
\caption{\small{Some angle and frame definitions: the  $q$-dependent product frame $(e_i(q))$, the eigenframe $(u_i)$ for the collision map $C_q$ at a collision configuration $q\in \partial M$,  the characteristic angle $\beta$ (a function of the mass distribution of the disc), and the angles $\phi(q,v)$ and $\psi(q,v)$. }} 
\label{eigenframe}
\end{center}
\end{figure}

\begin{definition}[Special orthonormal frames]\label{frames} For any given $\xi=(q,v)\in N$ we refer to   $$(e_1(q), e_2(q), e_3(q)), \ \ (u_1(q), u_2(q), u_3(q)), \ \ 
(w_1(\xi), w_2(\xi), w_3(\xi))$$ as the {\em product frame}, the {\em eigenframe}, and the {\em wavefront frame}, respectively.
\end{definition}

\section{Period-$2$  orbits}\label{periodic section}
It appears  to be  a harder problem in general to show the   existence of periodic orbits for no-slip billiards than it is for standard billiard systems in dimension $2$, despite numerical evidence that such points are common. A few  useful observations can still be made for specific shapes of $\mathcal{B}$. We begin here with the general description of period-$2$ orbits. The reader should bear in mind that, when we represent billiard orbits 
in figures such as \ref{near_periodic}, we often 
draw their projections on the plane, even though periodicity refers to a property of orbits in the $3$-dimensional reduced phase space. 

\begin{figure}[htbp]
\begin{center}
\includegraphics[width=4.5 in]{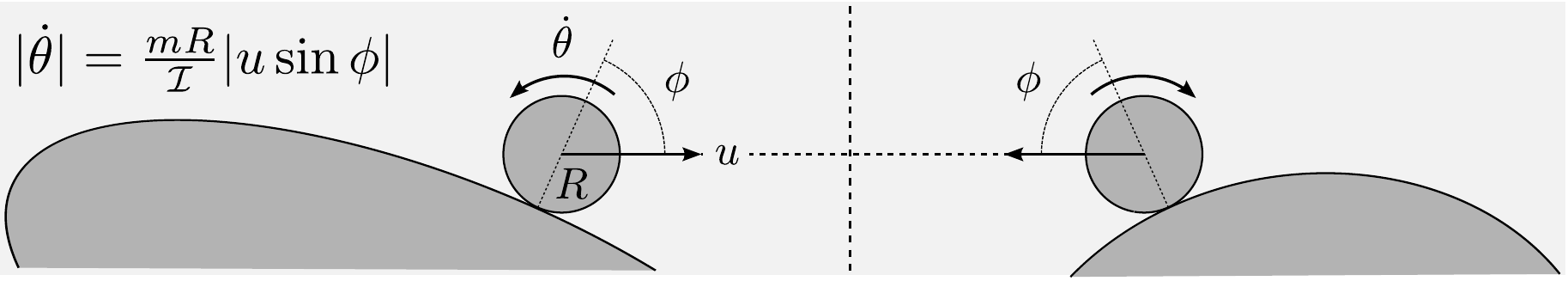}\ \ 
\caption{\small{Period $2$ orbit. The indicated parameters are: the disc's mass $m$, its moment of inertia  $\mathcal{I}$, and radius $R$. The velocity of the center of mass is $u$ and the angular velocity is $\dot{\theta}$.} }
\label{period2}
\end{center}
\end{figure}

Let $\xi=(q,v)$ be the initial state of a periodic orbit of period $2$, $\tilde{\xi}=(\tilde{q}, \tilde{v})=T(\xi)$, and $t$ the time of free flight between collisions. Then, clearly,
$$(q,v)=\left(\tilde{q} + tC_{\tilde{q}}v, C_q\tilde{v}\right) =\left(q+ t(v+C_{\tilde{q}}v), C_qC_{\tilde{q}}v\right)$$  
so that $C_{\tilde{q}}v=-v$ and $v=C_qC_{\tilde{q}}v$. Because $v$ and $u_1(q)$ (respectively, $u_1(\tilde{q})$) are eigenvectors for different eigenvalues  of  the orthogonal map $C_q$ (respectively, $C_{\tilde{q}}$), $v$ is perpendicular to both $u_1(q)$ and $u_1(\tilde{q})$. It follows from   (\ref{u in e}) that $u_1(q)\cdot e_1=u_1(\tilde{q})\cdot e_1$.
Thus the projection of $e_1$ to $v^\perp$ is proportional to $u_1(q)+u_1(\tilde{q})$. By the definition of the wavefront vector $w_1(\xi)$ (and the angle $\phi$, cf. Figure \ref{eigenframe}) we have
$$w_1(\xi)=w_1(\tilde{\xi}) =\frac{u_1(q)+u_1(\tilde{q})}{|u_1(q)+u_1(\tilde{q})|}=\frac{u_1(q)+u_1(\tilde{q})}{2\sqrt{1-\cos^2(\beta/2)\cos^2\phi}}. $$
Now observe that $u_1(\tilde{q})-u_1(q)$ is perpendicular to $u_1(q)+u_1(\tilde{q})$. It follows from this remark and a glance at Figure \ref{eigenframe} (to determine the orientation of the vectors) that 
$$w_2(\xi)=-w_2(\tilde{\xi})= \frac{u_1(\tilde{q})-u_1(q)}{|u_1(\tilde{q})-u_1(q)|}=\frac{u_1(\tilde{q})-u_1(q)}{2 \cos(\beta/2)\cos\phi}.$$
Notice, in particular, that $v$ is a positive multiple of $u_1(q)\times u_1(\tilde{q})$. 
An elementary calculation starting from this last observation gives $v$ in terms of the product frame: 
\begin{equation*}
v =\frac{\cos(\beta/2)\sin\phi e_1 + \sin(\beta/2)\left[ \sin\phi e_2(q) +  \cos\phi e_3(q)\right]}{\sqrt{1-\cos^2(\beta/2)\cos^2\phi}}.
\end{equation*}
A more physical description of the velocity $v$ of a period-$2$ orbit is shown in Figure \ref{period2} in terms of the moment of inertia $\mathcal{I}$.

 Equally elementary computations yield   the collision map $C_q$ in the wavefront frame at $q$, for a period-$2$ state $\xi=(q,v)$.  
We register this here  for later use. To shorten the equations we write $c_{\beta/2}=\cos(\beta/2)$ and $c_\phi = \cos\phi$.
\begin{align}\label{Cwavefront}
\begin{split}
C_q w_1(\xi) &= \left(1- 2 c_{\beta/2}^2 c^2_\phi\right) w_1(\xi) -2c_{\beta/2} c_\phi \sqrt{1-c^2_{\beta/2} c^2_\phi} w_2(\xi)\\
C_q w_2(\xi) & = -2 c_{\beta/2} c_\phi\sqrt{1-c^2_{\beta/2} c^2_\phi} w_1(\xi) - \left(1-2c^2_{\beta/2}  c^2_\phi\right) w_2(\xi)\\
C_q w_3(\xi) & = -w_3(\xi).
\end{split}
\end{align}

The following easily obtained inner products will also be needed later.
\begin{align}\label{inner}
\begin{split}
u_1(\tilde{q})\cdot u_1({q})&=1-2\cos^2(\beta/2)\cos^2\phi\\
 w_1(\xi)\cdot u_1(q)&=\sqrt{1-\cos^2(\beta/2)\cos^2\phi}\\
  w_2(\xi)\cdot u_1(q)&= -\cos(\beta/2)\cos\phi. 
  \end{split}
  \end{align}

Figure \ref{sinai}  shows (two copies of the fundamental domain of) the configuration manifold of the no-slip Sinai billiard. Here the  billiard table is the complement of a circular scatterer in a two-dimensional torus and $M$ is the cartesian  product of the latter with a one-dimensional torus. Notice that there is a whole one-parameter family of initial conditions giving period-$2$ orbits, parametrized by the angle $\phi$. We obtain infinitely many such families by choosing different pairs of fundamental domains.

\begin{figure}[htbp]
\begin{center}
\includegraphics[width=3.5 in]{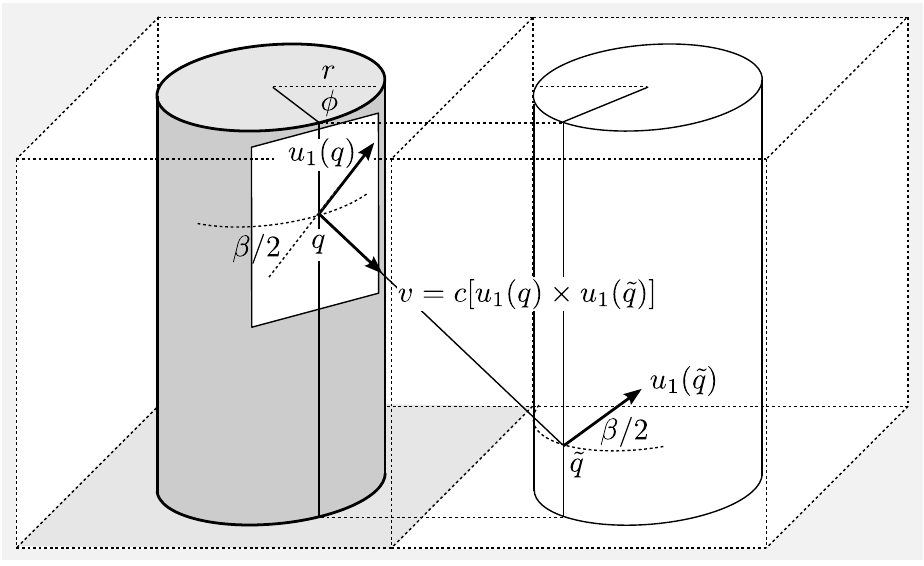}\ \ 
\caption{\small{The figure shows two fundamental domains of the no-slip Sinai billiard and an initial velocity $v$ of a periodic orbit with period $2$. This trajectory lies in a one-parameter family of period-$2$ trajectories parametrized by the angle $\phi$.}}
\label{sinai}
\end{center}
\end{figure}

We will return to the no-slip Sinai billiard in Section \ref{curved section}.

\section{The differential of the no-slip billiard map}\label{sec:differential}
Mostly, in this section,  we write    $\langle u,v\rangle$ instead of  $u\cdot v$ for the standard inner product of $\mathbb{R}^3$.
Let $q(s)$ be a smooth curve in $\partial M$ such that $q(0)=q$ and $q'(0)=X\in T_q(\partial M)$. Define
$$ \omega_q(X):= \left. \frac{d}{ds}\right|_{s=0} \sigma(q(0))^{-1}\sigma(q(s))\in \mathfrak{so}(3)$$
where  $\mathfrak{so}(3)$ is  the space of antisymmetric $3\times 3$ matrices (the Lie algebra of the rotation group) and $\sigma(q):=\sigma_q$ is the product frame.
 As $e_1$ is a parallel field and $\omega_q(X)$ is antisymmetric
we have $\omega_q(X)_{ij}=0$ except possibly for $(i,j)=(2,3)$ and $(3,2)$. 
Denoting by $D_X$ directional derivative  of vector fields along $X$ at $q$, 
$$\omega_q(X)_{23}=\epsilon_2 \cdot \left[ \left. \frac{d}{ds}\right|_{s=0} \sigma(q(0))^{-1}\sigma(q(s))\epsilon_3\right]= \left\langle e_2(q), D_X e_3\right\rangle= \left\langle e_2(q), X\right\rangle \left\langle e_2(q), D_{e_2(q)}e_3\right\rangle$$
since $D_{e_1}e_3=0$. The inner product $\kappa(q):= \left\langle e_2(q), D_{e_2(q)}e_3\right\rangle$ is the geodesic curvature of the boundary of $\mathcal{B}$ at $\bar{q}$, where $\bar{q}$ is the base point of $q$ in
$\partial \mathcal{B}$.  Thus 
\begin{equation}\label{curvature} \omega_q(X)=\kappa(q) \langle e_2(q), X\rangle \mathcal{A}\end{equation}
where $$\mathcal{A}= \left(\begin{array}{ccc}0 & 0 & 0 \\0 & 0 & 1 \\0 &\!\! -1 & 0\end{array}\right).$$

Given  vector fields $\mu, \nu$, we define   $\mu\odot \nu$ as the map 
\begin{equation}\label{odot} (q, v)\mapsto (\mu\odot\nu)_qv:=\langle \mu_q, v\rangle \nu_q + \langle \nu_q, v\rangle \mu_q.\end{equation}

\begin{lemma}\label{curvlemma}
The directional derivative of   $C$ along $X\in T_q(\partial M)$ is  
$$ D_XC= \kappa(q) \left\langle e_2(q), X\right\rangle \mathcal{O}_q$$
where $\mathcal{O}_q=\sigma_q \mathcal{O}\sigma_q^{-1}$, 
 $$\mathcal{O}:=[\mathcal{A},\mathcal{C}]=2\cos(\beta/2)\left(\begin{array}{ccc}0 & 0 & \ \, \sin(\beta/2) \\0 & 0 & -\cos(\beta/2) \\\sin(\beta/2) & -\cos(\beta/2) & 0\end{array}\right) $$
and $\mathcal{C}$ was defined above in (\ref{collision_map}). Furthermore, 
$\mathcal{O}_q=2\cos(\beta/2) (u_1\odot e_3)_q$ and
$$ D_X C = 2 \cos(\beta/2)\kappa(q) \langle X, e_2\rangle_q (u_1\odot e_3)_q.$$
\end{lemma}
\begin{proof}
Notice that
$ 0 = D_X I= D_X(\sigma^{-1}\sigma)= (D_X\sigma^{-1})\sigma+ \sigma^{-1}D_X\sigma.$ Thus
$$D_X\sigma^{-1}= -\sigma^{-1} \left(D_X\sigma\right) \sigma^{-1}.$$
Therefore,
$$ D_XC=(D_X\sigma) \mathcal{C} \sigma^{-1} + \sigma \mathcal{C} D_X\sigma^{-1}=\sigma\left[\sigma^{-1}D_X\sigma\right] \mathcal{C} \sigma^{-1} - \sigma \mathcal{C}\left[\sigma^{-1} D_X\sigma\right] \sigma^{-1}= \sigma[\omega(X),\mathcal{C}]\sigma^{-1}.$$
The first claimed expression for $D_XC$ is now a consequence of  Equation \ref{curvature}. A simple computation also gives, for any given $v\in \mathbb{R}^3$,  
\begin{equation}\label{O} \sigma(q) \mathcal{O}\sigma(q^{-1})v=
2\cos(\beta/2)(e_3\odot u_1)_{{q}}v\end{equation}
yielding the second expression for $D_XC$.
\end{proof}

It is convenient to define the following two projections.
Let $\xi=(q,v)\in N^\pm$.  The space $T_\xi N^\pm$ decomposes as a direct sum
$T_\xi N^\pm=H_\xi\oplus V_\xi$
where 
$$ H_\xi=T_qN=\{X\in \mathbb{R}^3: X\cdot e_3(q)=0\} \text{ and } V_\xi=v^\perp=\{Y\in \mathbb{R}^3:Y\cdot v=0\}.$$
(Recall that $N:=N^+$.) We refer to these as the {\em horizontal} and {\em vertical} subspaces of $T_\xi N^\pm$. 
We use the same symbols  to denote the projections
$H_\xi: \mathbb{R}^3\rightarrow T_q(\partial M)$ and $V_\xi:\mathbb{R}^3\rightarrow v^\perp$ defined by
$$ H_\xi Z:= Z- \frac{\langle Z, e_3(q)\rangle}{\langle v, e_3(q)\rangle} v, \ \ V_\xi Z:=Z- \langle Z, v\rangle v.$$ 
Notice that for $\xi=(q,v)\in N^\pm$ and $Z\in \mathbb{R}^3$ 
$$
\langle e_2(q), H_\xi Z\rangle
= \frac{\langle  Z, e_2(q)\rangle \langle v, e_3(q)\rangle- \langle Z, e_3(q)\rangle   \langle v, e_2(q)\rangle}{\langle v, e_3(q)\rangle}=\frac{\langle v\times e_1(q), Z\rangle}{\langle v, e_3(q)\rangle}.
$$
Also observe that $v\times e_1= |\overline{v}| w_2(\xi)$, where $w_2$ is the second vector in the wavefront frame (cf. Definition \ref{frames}) and $\overline{v}$ is the orthogonal projection of $v$
to the plane perpendicular to $e_1$. Thus, denoting by $\phi(\xi)$ the angle between $\overline{v}$ and $e_3(q)$ (this is the same $\phi$ as in Figures \ref{eigenframe}, \ref{period2} and \ref{sinai})
\begin{equation}\label{cross}
\langle e_2(q), H_\xi Z\rangle=\frac{1}{\cos\phi(\xi)} \langle w_2(\xi), Z\rangle. 
\end{equation}

 Let $q\in \partial M$, $v=v_-\in N_q^-$, $v_+:=C_qv_-\in N_q^+$, 
$\xi=\xi_-=(q,v_-)$, $\xi_+=(q,v_+)$. Define \begin{equation}\label{Lambda}\Lambda_\xi:= V_{\xi_+} H_{\xi_-} : v_-^\perp\rightarrow v_+^\perp.\end{equation}
Clearly $\Lambda_{\xi}$ is  defined on all of $\mathbb{R}^3$, not only on $v_-^\perp$, but we are particularly interested in
its restriction to the latter subspace.

Let $\xi=(q,v)$ be a point contained in a neighborhood of $N$ where $T$ is defined and differentiable. Set $\tilde{\xi}= T(\xi)$.
We wish to describe
$dT_\xi: T_\xi N\rightarrow T_{\tilde{\xi}}N$. Let $\xi(s)=(q(s), v(s))$ be a differentiable curve in $N$ with $\xi(0)=\xi$ and  define $$X:=q'(0)\in T_qN,\ \  Y:=v'(0)\in v^\perp.$$
Then $\tilde{\xi}(s)= T(\xi(s))=(\tilde{q}(s), \tilde{v}(s))\in N$ where $\tilde{q}(s)=q(s) + t(s) v(s)$ and $\tilde{v}(s)=C_{\tilde{q}(s)} v(s)$.
From the equality
$ \langle  \tilde{q}'(0), e_3(\tilde{q})\rangle=0$  it follows that 
$$t'(0)= - \frac{\langle X+tY, e_3(\tilde{q})\rangle}{\langle v, e_3(\tilde{q})\rangle}.$$
Consequently,  $\tilde{X}:=\tilde{q}'(0)\in T_{\tilde{q}}N$ and $\tilde{Y}:=\tilde{v}'(0)\in \tilde{v}^\perp$ satisfy
$$ \tilde{X}=X+tY - \frac{\langle X+tY, e_3(\tilde{q})\rangle}{\langle v, e_3(\tilde{q})\rangle}v= H_{\tilde{\xi}_-}(X+tY)$$
and $$ \tilde{Y}=C_{\tilde{q}}Y +\left[\left. \frac{d}{ds}\right|_{s=0} C_{\tilde{q}(s)}\right]v=C_{\tilde{q}} Y + \kappa(\tilde{q}) \langle e_2(\tilde{q}), \tilde{X}\rangle \mathcal{O}_{\tilde{q}} v$$ 
where we have used Lemma \ref{curvlemma}. From the same lemma,
$$
\sigma(\tilde{q})\mathcal{O}\sigma(\tilde{q})^{-1} v=-2\cos(\beta/2)(\nu\odot u_1)_{\tilde{q}}v.
$$
Thus 
\begin{align}
\begin{split}
\tilde{X}&= H_{\tilde{\xi}_-}(X+tY)\\
\tilde{Y}&=C_{\tilde{q}}Y -2\cos(\beta/2) {\kappa(\tilde{q})}\left\langle e_2(\tilde{q}), H_{\tilde{\xi}_-}(X+tY)\right\rangle
 (\nu\odot u_1)_{\tilde{q}}v. 
 \end{split}
\end{align}

As already noted,  $T_\xi N^+=T_q(\partial M)\oplus v^\perp$. By using the projection $V_\xi:T_q(\partial M)\rightarrow v^\perp$ introduced earlier we
may identify $T_\xi N^+$ with the sum $v^\perp\oplus v^\perp$. In this way $dT_\xi$ is regarded as a map from $v^\perp\oplus v^\perp$ to
$\tilde{v}^\perp\oplus \tilde{v}^\perp.$

\begin{proposition}\label{differential}
   Let $T:N\rightarrow N$ be the billiard map,
 $\xi=(q,v)\in N$  and  $(\tilde{q},\tilde{v})=\tilde{\xi}=T(\xi)$, where $\tilde{q}=q+tv$, and $\tilde{v}=C_{\tilde{q}}v$. 
Under the identification of the tangent space $T_\xi N$ with $v^\perp\oplus v^\perp$ as indicated just above, we may regard  the differential
  $dT_\xi$ as a linear map from 
  $v^\perp\oplus v^\perp$ to
$\tilde{v}^\perp\oplus \tilde{v}^\perp.$ Also recall from (\ref{cross})   the definition of  $\Lambda_{\tilde{\xi}}:v^\perp\rightarrow \tilde{v}^\perp$. Then $dT_\xi :T_\xi N\rightarrow T_{\tilde{\xi}} N$ is given by
$$
\left(\begin{array}{c}X \\Y\end{array}\right)\mapsto
\left(\begin{array}{l}
\Lambda_{\tilde{\xi}}(X+tY)\\
C_{\tilde{q}}Y +2\cos(\beta/2)\kappa(\tilde{q})\frac{ (e_3\odot u_1)_{\tilde{q}}v}{\cos\phi(\tilde{q},v)}
\langle w_2(\xi),X+tY\rangle 
\end{array}\right)
$$
where $\cos\phi(\tilde{q},v)= \langle \overline{v}/|\overline{v}|, e_3(\tilde{q})\rangle$ and $\overline{v}$ is the orthogonal projection of $v$ to $e_1^\perp$.
\end{proposition}
\begin{proof}
This is a consequence of the  preceding remarks and definitions. 
\end{proof}

\begin{corollary}\label{cor}
If $\xi=(q,v)$ is periodic of period $2$,  then $C_{\tilde{q}}v=-v$, $\langle v, u_1(\tilde{q})\rangle=0$, and the map of Proposition \ref{differential} reduces to
$$
\left(\begin{array}{c}X \\Y\end{array}\right)\mapsto
\left(\begin{array}{l}
X+tY\\
C_{\tilde{q}}Y +2\cos(\beta/2) {\kappa(\tilde{q})}\frac{\cos\psi(\tilde{q},v)}{\cos \phi(\tilde{q},v)}  \left\langle w_2(\xi), X+tY\right\rangle   u_1(\tilde{q})\end{array}\right). 
$$
where $\cos\psi(\tilde{q},v):=\langle v, e_3(\tilde{q}) \rangle$, $\cos\phi(\tilde{q},v)= \langle \overline{v}/|\overline{v}|, e_3(\tilde{q})\rangle$.
\end{corollary}
\begin{proof}
Clearly, $C_{\tilde{q}}v=-v$, whence   $\langle v, u_1(\tilde{q})\rangle=0$ and
 $(e_3\odot u_1)_{\tilde{q}}v=\langle e_3(\tilde{q}), v\rangle u_1(\tilde{q})$. Also notice that $\Lambda_{\tilde{\xi}}Z=Z$ whenever  $\langle Z, v\rangle=0$. 
The corollary follows.
\end{proof}

\section{Measure invariance and time reversibility}\label{measure reversible} It will be seen below that the no-slip billiard map does not preserve the natural symplectic form on $N$, so these mechanical systems are not  Hamiltonian. Nevertheless, the canonical billiard measure derived from the symplectic form (the Liouville measure)  is invariant and the system is time reversible, so some of the good features of Hamiltonian systems are still present. (See, for example,  \cite{sevryuk, sevryuk91} where  a KAM theory is developed for reversible systems.)

Recall that the invertible map $T$ is said to be {\em reversible} if there exists an involution $\mathcal{R}$ such that $$\mathcal{R}\circ  T\circ \mathcal{R}=T^{-1}.$$ 

In order to see that the no-slip billiard map $T$ is reversible we first define the following maps:
$\Phi: (q,v)\mapsto (q+tv,v)$, where $t$ is the time of free motion of the trajectory with initial state $(q, v)$, so that $q, q+tv\in \partial M$; the collision map $C:N\rightarrow N$ given by
$C(q,v)=(q, C_qv)$; and the flip map $J:(q,v)\mapsto (q,-v)$ where $q\in \partial M$ and $v\in \mathbb{R}^3$. Recall that $T=C\circ \Phi$.
Now set $\mathcal{R}:=J\circ C=C\circ J$. It is clear (since $C_q$ is an involution by Proposition \ref{classification}) that $\mathcal{R}^2 = I$ and that $J\circ \Phi\circ J = \Phi^{-1}$. Therefore,
$$\mathcal{R}\circ T\circ \mathcal{R}= J\circ C^2\circ \Phi\circ J\circ C=  J\circ \Phi\circ J\circ C= \Phi^{-1}\circ C = (C\circ \Phi)^{-1}= T^{-1}.  $$

Notice that if $L:V\rightarrow V$ is a reversible isomorphism of a vector space $V$ with time reversal map $\mathcal{R}:V\rightarrow V$ (so that $\mathcal{R}\circ L\circ \mathcal{R}=L^{-1}$) then for any eigenvalue $\lambda$ of $L$ associated to eigenvector $u$, $1/\lambda$ is also an eigenvalue for the eigenvector $\mathcal{R}u$, as easily checked.  These simple observations have the following useful corollary.

\begin{proposition}
Let $\xi\in N$ be a periodic point  of period $k$ of the no-slip billiard system and let $\lambda$ be an eigenvalue of the differential map $dT^k_\xi: T_\xi N\rightarrow T_\xi N$ corresponding to  eigenvector $u$.
Then $1/\lambda$ is also an eigenvalue of $dT^k_\xi$ corresponding to eigenvector $\mathcal{R} u$, where $\mathcal{R}$ is the composition of the collision map $C$ and the flip map $J$. Furthermore,   $e_1$ (see Definition \ref{frames}) is always an eigenvector of $dT_\xi$ and all its powers, corresponding to the eigenvalue $1$.
\end{proposition} 

We now turn to invariance of the canonical measure.
The canonical $1$-form $\theta$ on $N$ is defined by
$$\theta_\xi(U):=v\cdot X $$ for  $\xi=(q,v)\in N$ and $U=(X,Y)\in T_qN\oplus v^\perp = T_\xi N$.  Its differential $d\theta$ is
a symplectic form on $N\cap \{v\in T_q(\partial M):|v|=1\}^c$ and $\Omega=d\theta\wedge d\theta$ is the canonical volume form
on this same set.  In terms of horizontal and vertical components of vectors in $TN$, the symplectic form is expressed as
$$d\theta(U_1, U_2)= Y_1\cdot X_2  - Y_2\cdot X_1$$
where $U_i=(X_i,Y_i)$.  An elementary computation shows that the measure on $N$ associated to $\Omega$ is given by
\begin{equation}\label{liouville}|\Omega_\xi|= v\cdot \nu(q) \, dA_{\partial M}(q)\, dA_{N}(v) \end{equation}
where $\nu(q):=e_3(q)$,  $dA_{\partial M}(q)$ is the area measure on $\partial M$, and $dA_{N}(v)$ is the area measure on the hemisphere $N_q=\{v\in \mathbb{R}^3:v\cdot \nu(q)>0\}$.
\begin{proposition}
The canonical $4$-form $\Omega$ on $N$ transforms under the no-slip billiard map as $T^*\Omega=-\Omega$. In particular, the
associated measure $|\Omega|$, shown explicitly in Equation \ref{liouville}, is invariant under $T$. 
\end{proposition}
\begin{proof} Let $u$ be a vector field on $\partial M$ and introduce the one-form $\theta^u$  on $N$ given by
 $$ \theta^u_\xi(U):= (v\cdot u(q)) (u(q)\cdot X)$$
for $\xi=(q,v)$ and $U=(X,Y)$.  Taking $u$ to be each of the vector fields $u_1, u_2$ we obtain the $1$-forms $\theta^{u_1}$ and $\theta^{u_2}$.
As $v=(v\cdot u_1) u_1 + (v\cdot u_2) u_2 + (v\cdot \nu) \nu$ and $X\cdot \nu=0$, we have
$$ \theta=\theta^{u_1}+\theta^{u_2}.$$
The no-slip collision map $C$ acts on $u=\theta^{u_i}$ as follows: For  $U=(X, Y)\in T_q(\partial M)\oplus v^\perp$,
$$(C^*\theta^u)_\xi(U)= (C_q(v)\cdot u(q))(u(q)\cdot X)= (v\cdot C_q(u(q)))(u(q)\cdot X). $$
It follows that
$$ C^*\theta^{u_1}=\theta^{u_1}, \ \ C^*\theta^{u_2}= -\theta^{u_2}.$$
 
We now compute the differentials $d\theta^{u}$ for $u=u_1, u_2$. Observe that $\theta^{u}=f^{u}(\xi) (\pi^* u^\flat)$, where $f^\xi$ is the function
on $N$ defined by $f^u(\xi):=v\cdot u(q)$ and $\pi^*u^\flat$ is the pull-back under the projection map $\pi:N\rightarrow \partial M$ of the
$1$-form $u^\flat$ on $\partial M$ given by $u_q^\beta(X)=u(q)\cdot X$.  Thus
$$d\theta^u= df^u\wedge (\pi^*u^\flat) + f^u\pi^* du^\flat. $$
A simple calculation gives
$$ df^u_\xi(X,Y)= v\cdot (D_Xu) + u(q)\cdot Y.$$
The vector field $u=u_i$ is parallel on $\partial M$.  In fact, its derivative in direction $X\in T_q(\partial M)$ only has component in the normal direction, given by
$$D_Xu=\kappa(q) (X\cdot e_2(q)) (u(q)\cdot e_2(q)) \nu(q). $$
Omitting the dependence on $q$, we have
$$df^u_\xi(X,Y)= \kappa(q)(X\cdot e_2)(u\cdot e_2)(v\cdot \nu) + u\cdot Y. $$
Another simple calculation gives
$$ du_q^\flat(X_1,X_2)=(D_{X_1}u)\cdot X_2 - (D_{X_2}u)\cdot X_1=0$$
so $d\theta^u=df^u\wedge \pi^*u^\flat$.
Explicitly,
$$ d\theta^u(U_1, U_2)=(u\cdot Y_1)(u\cdot X_2)-(u\cdot Y_2)(u\cdot X_1)- \kappa(q) (v\cdot\nu)(u\cdot e_1)(u\cdot e_2)\omega(X_1, X_2)$$
where 
$$ \omega(X_1, X_2):=(e_1\cdot X_1)(e_2\cdot X_2)-(e_2\cdot X_1)(e_1\cdot X_2).$$
Notice  that $\omega$ is the area form  on $\partial M$.
A convenient way to express $d\theta^u$ is as follows. Define the $1$-form $\tilde{u}$ on $N$ by $\tilde{u}_\xi(U)=u(q)\cdot Y$, where
$U=(X,Y)\in T_\xi N$, and the function $g^u(\xi):=- \kappa(q)(v\cdot\nu)(u\cdot e_1)(u\cdot e_2)$.  These extra bits of notation now allow us to write
$$d\theta_\xi^u= g^u(\xi) (\pi^* \omega) + \tilde{u}\wedge (\pi^* u^\flat). $$
The main conclusion we wish to derive from these observations is that $d\theta^u\wedge d\theta^u=0$. This is the case because, as $\dim (\partial M)=2$,
 we must have  $\omega^2=0$ and $\omega\wedge u^\flat=0$.
Therefore,
$$\Omega:= d\theta\wedge d\theta= (d\theta^{u_1}+d\theta^{u_2})\wedge (d\theta^{u_1}+d\theta^{u_2})= 2 d\theta^{u_1}\wedge d\theta^{u_2}.$$
Finally,
$$ C^*\Omega = 2 d(C^* \theta^{u_1})\wedge d(C^*\theta^{u_2})= - 2 d\theta^{u_1}\wedge d\theta^{u_2}=-\Omega.$$
The forms  $d\theta$ and $\Omega$ are invariant under the geodesic flow and under the map it induces on $N$.  As $T$ is the composition of this map and $C$,
the proposition is established.
 \end{proof}

\section{Wedge billiards}\label{wedge section}
We set the following conventions for a wedge table with corner angle $2\phi$. See Figure \ref{wedge}. (This is the same $\phi$ that has appeared before in previous figures.) The boundary planes of the configuration manifold are denoted $\mathcal{P}_1$ and
$\mathcal{P}_2$. The orthonormal vectors of the constant product frame on plane $\mathcal{P}_i$ are $e_{1,i}, e_{2,i}, e_{3,i}=\nu_i$ for $i=1,2$ where
$$\begin{tabu}{lll}
 e_{1,1}=\left(\begin{array}{c}0 \\0 \\1\end{array}\right), &
e_{2,1}=\left(\begin{array}{c}\cos\phi \\ \!\!\!\!-\sin\phi \\0\end{array}\right), & 
e_{3,1}=\left(\begin{array}{c}\sin\phi \\ \cos\phi \\0\end{array}\right),\\
&  &\\
  e_{1,2}=\left(\begin{array}{c}0 \\0 \\1\end{array}\right), &
e_{2,2}=-\left(\begin{array}{c}\cos\phi \\ \sin\phi \\ 0\end{array}\right),  &
e_{3,2}=\left(\begin{array}{c}\sin\phi \\ \!\!\!\!-\cos\phi\\ 0\end{array}\right).
\end{tabu}$$

Let $\sigma_i:\mathbb{R}^3\rightarrow T_q\oplus \mathbb{R}\nu_i$ be  the constant orthogonal map such that $\sigma_i \epsilon_j=e_{j, i}$, where
 $\epsilon_i$, $i=1, 2, 3$,  is our   notation for the standard basis vectors in $\mathbb{R}^3$.
 Let $$u_{1,i}=\sin(\beta/2) e_{1,i} - \cos(\beta/2) e_{2,i}, \ \ u_{2,i}=\cos(\beta/2) e_{1,i}+\sin(\beta/2)e_{2,i}, \ \ u_{3,i}=e_{3,i}=\nu_i$$  be the    eigenvectors of the no-slip reflection map associated to the plane $\mathcal{P}_i$ and set $\zeta_i\epsilon_j:= u_{j,i}$.   For easy reference we record their matrices here:
 $$\zeta_i=\left(\begin{array}{ccc}(-1)^i\cos(\beta/2)\cos\phi & -(-1)^i\sin(\beta/2)\cos\phi & \sin\phi \\ \cos(\beta/2)\sin\phi & -\sin(\beta/2)\sin\phi & -(-1)^i\cos\phi \\\sin(\beta/2) & \cos(\beta/2) & 0\end{array}\right). $$
The initial velocity $v$ for the period-$2$ trajectory  points in the direction of $u_{1,2}\times u_{1,1}$ and  is given
by
$$v=\frac1{\sqrt{1-\cos^2(\beta/2)\cos^2\phi}}\left(\begin{array}{c}0 \\\sin(\beta/2) \\ \cos(\beta/2)\sin\phi\end{array}\right). $$
This periodic trajectory connects the points $q_1\in \mathcal{P}_1$ and $q_2\in \mathcal{P}_2$. Any such pair of points
can be written as
$$ q_1=a\left(\begin{array}{c}\sin(\beta/2)\cos\phi \\\!\!\!\! -\sin(\beta/2)\sin\phi \\b-\cos(\beta/2)\sin^2\phi\end{array}\right), \ \ q_2=a\left(\begin{array}{c}\sin(\beta/2)\cos\phi \\  \sin(\beta/2)\sin\phi \\ b+\cos(\beta/2)\sin^2\phi\end{array}\right)$$
where $a, b\in \mathbb{R}$, $a>0$.  In what follows we assume without  loss of generality that $a=1$ and $b=0$.  Thus
$$q_i = \left(\sin(\beta/2)\cos\phi, (-1)^i\sin(\beta/2)\sin\phi, (-1)^i\cos(\beta/2)\sin^2\phi\right)^t. $$

\begin{figure}[htbp]
\begin{center}
\includegraphics[width=3.5 in]{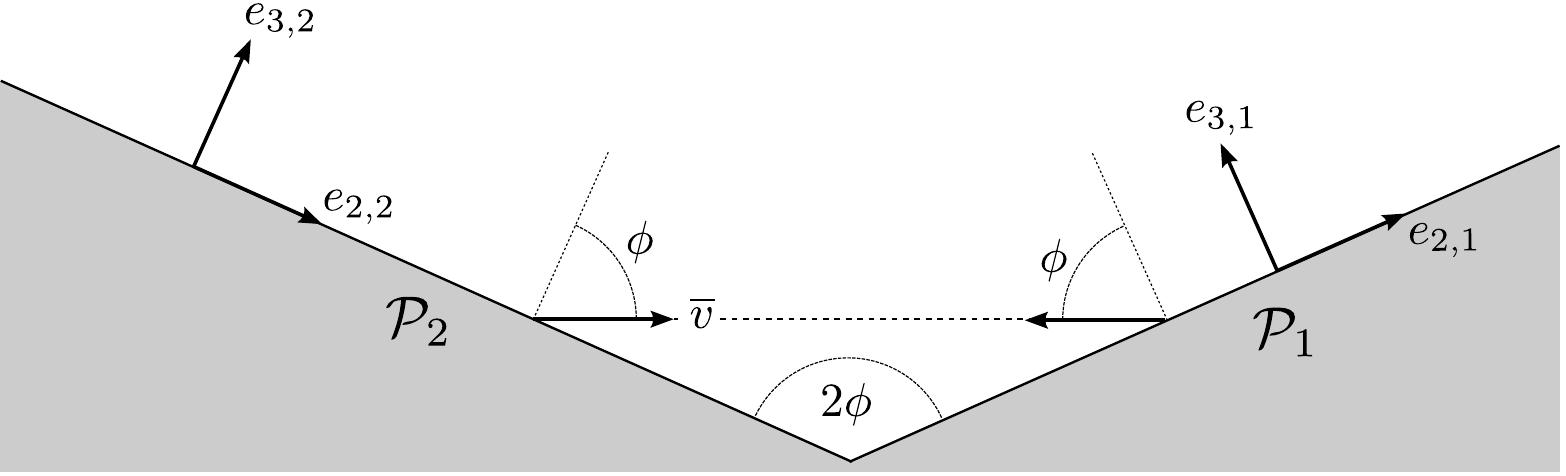}\ \ 
\caption{\small{Some notation specific to the wedge billiard table. The $\mathcal{P}_i$ are the half-plane components of the boundary of the configuration manifold.}}
\label{wedge}
\end{center}
\end{figure}

Let $S^\pm_i=\{v\in \mathbb{R}^3: |v|=1, \pm v\cdot \nu_i> 0 \}$.  
The collision  maps $C_i:S^-_{i}\rightarrow S^+_{i}$, $i=1,2$ are given by the matrices  
$$ C_i=\sigma_i\mathcal{C}\sigma_i^{-1}=
\zeta_i \left(\begin{array}{ccc}1 & 0 & 0 \\0 & \!\!\! -1 & 0 \\0 & 0 &\!\!\! -1\end{array}\right)\zeta_i^{-1}$$
where $\mathcal{C}$ was defined in \ref{collision_map}.
We now  introduce coordinates on $\mathcal{P}_i\times S^+_i$ as follows. Let  ${S}_+^2=\{z\in \mathbb{R}^3: |z|= 1, z_3>0\}$ and define $\Phi_i:\mathbb{R}^2\times S_+^2\rightarrow \mathcal{P}_i\times S^+_i$ by
$$ \Phi_i(x,y)=\left(q_i+ x_1 u_{1,i}+x_2 u_{2,i}, y_1 u_{1,i}+y_2u_{2,i}+y_3u_{3,i}\right).$$
Regarding $x\in \mathbb{R}^2$ as $(x, 0)\in \mathbb{R}^3$,  we may then write
$$ \Phi_i(x,y)= \left(q_i+\zeta_ix, \zeta_iy\right).$$

Clearly, the billiard map is not defined on all of $\bigcup_i\mathcal{P}_i\times S^+_i$ since those initial velocities not pointing
towards the other plane will escape to infinity, but we are interested in the behavior of the map
on a neighborhood of the periodic point $\xi_i=(q_i, v_i)$, $v_i=-(-1)^{i}v$. The question of interest here is whether some open neighborhood of $\xi_i$
remains invariant under the billiard map.
It is easily shown that the coordinates of the state $\xi_i$ (of  the period-$2$ orbit at the plane $\mathcal{P}_i$) are 
$\Phi_i^{-1}(\xi_i)=(0,y_i)\in \mathbb{R}^2\times S^2_+$
where

$$y_i=\frac1{\sqrt{1-\cos^2(\beta/2)\cos^2\phi}}\left(0, (-1)^i\sin\phi,  \sin(\beta/2)\cos\phi\right)^t$$
Let $T_i:\mathcal{D}_i\subset \mathbb{R}^2\times S_+^2\rightarrow \mathbb{R}^2\times S_+^2$ be   the billiard map  
restricted to $\mathcal{P}_i\times S^+_i$ expressed in the coordinate system defined by $\Phi_i$.  
 Thus
$$ T_1=\Phi_2^{-1} T \Phi_1, \ \ T_2=\Phi_1^{-1}T\Phi_2$$ 
on their  domains $\mathcal{D}_i$.  We now find the explicit form of $T_i$. 
 Define  $\bar{i}=\begin{cases}1& \text{ if } i=2\\ 2 & \text{ if } i=1\end{cases}$ and
  orthogonal  matrices 
$ A_i := \zeta_{\bar{i}}^{-1} \zeta_i$ and
$S=\text{diag}(1,-1,-1)$, both in $SO(3)$. Also define 
$$\alpha:=2\sin\phi\sqrt{1-\cos^2(\beta/2)\cos^2\phi}.$$
Observe that 
$ \zeta_{\bar{i}}^{-1}C_{\bar{i}} \zeta_i= S A_i$.  
It is easily shown that
$$q_i-q_{\bar{i}}=-\alpha v_i, \ v_i=\zeta_i {y}_i, \  A_i {y}_i=-{y}_{\bar{i}}, \ SA_i y_i=y_{\bar{i}}.$$
In particular, $ \zeta^{-1}_{\bar{i}}(q_i-q_{\bar{i}})=-\alpha {y}_i$.
 Let  
  $Q:\mathbb{R}^3\times S^2_+\rightarrow \mathbb{R}^2$ be defined by
$$ Q(x,y):= x-\frac{x\cdot\epsilon_3}{y\cdot \epsilon_3}y.$$ 
Notice    that $Q(x,y)\cdot \epsilon_3=0$. 
We now have
\begin{equation}T_i:(x,y)\mapsto \left({Q\left(A_i(x-\gamma y_i), A_i{y}\right)}, SA_i {y}\right). \end{equation}

\begin{figure}[htbp]
\begin{center}
\includegraphics[width=2.5 in]{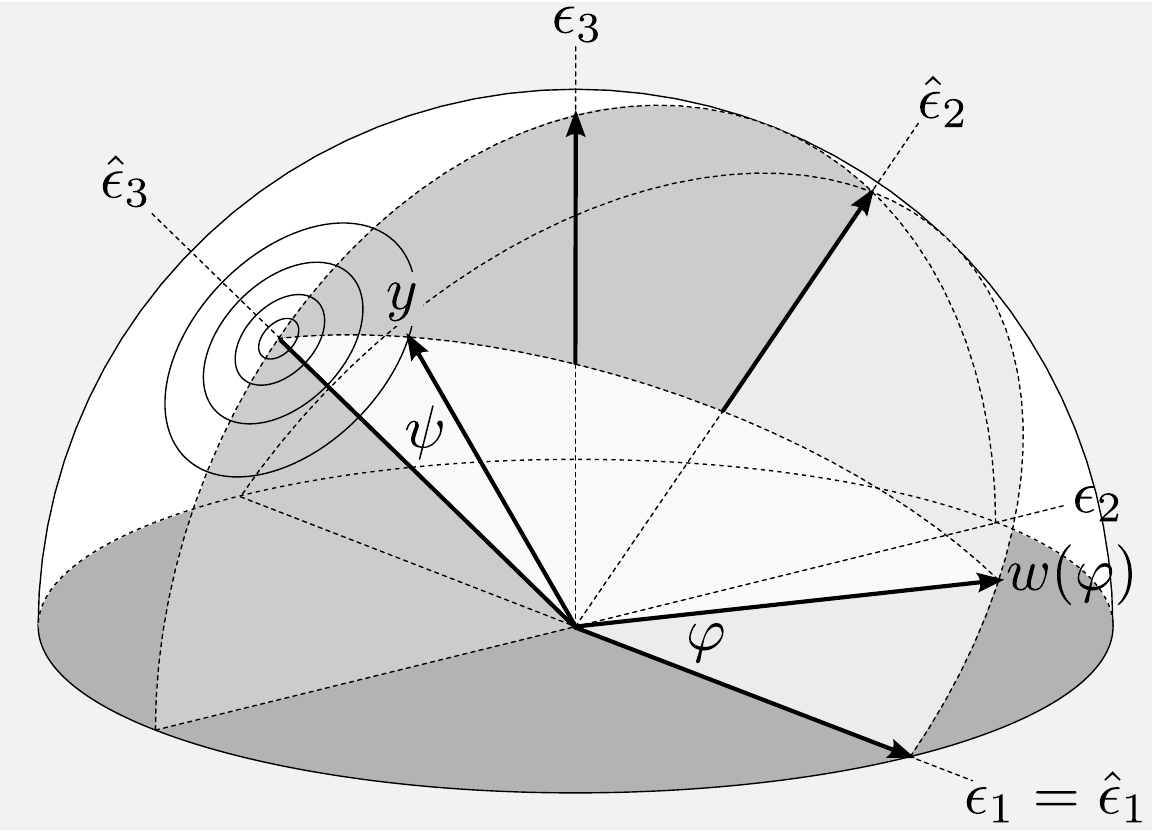}\ \ 
\caption{\small{The velocity factor  of orbits of the return billiard map $T_2T_1$  in coordinate system $\Phi_1$  lie in concentric circles with axis $y_1=\hat{\epsilon}_3$.
We use spherical  coordinates $\varphi$ and $\psi$ relative to the axis $\hat{\epsilon}_3$ to represent the velocity $y\in S^2_+$.
With respect to these coordinates, the return map sends $w(\varphi)$ to $w(\varphi+\theta)$, where $\theta$ is a function   of the wedge angle $\alpha$ and the characteristic
angle $\beta$ of the no-slip reflection.  }}
\label{directions}
\end{center}
\end{figure}

 For easy reference we record
  $$\alpha y_i=2\sin\phi\left(\begin{array}{c}0 \\(-1)^i\sin\phi \\\sin(\beta/2) \cos\phi\end{array}\right), \ \  S=\left(\begin{array}{crr}1 & 0 & 0 \\0 & -1 & 0 \\0 & 0 & -1\end{array}\right)$$
  and 
$$  A_2= A_1^t=\zeta^{-1}_1\zeta_2=\left(\begin{array}{ccc}1-2\cos^2(\beta/2)\cos^2\phi & -\sin\beta \cos^2\phi & \cos(\beta/2)\sin(2\phi) \\
-\sin\beta\cos^2\phi & 1-2\sin^2(\beta/2)\cos^2\phi & \sin(\beta/2)\sin(2\phi) \\
-\cos(\beta/2)\sin(2\phi) & -\sin(\beta/2)\sin(2\phi) & -\cos(2\phi)\end{array}\right).$$

Using  the notation $[z]_3:=z\cdot\epsilon_3$ and elementary computations based on the above gives:
\begin{proposition}
The return map in the coordinate system defined by $\Phi_1$  has the form
$$ T_2T_1(x,y)=\left(x+ [A_1(x-\alpha y_1)]_3 V(y), SA_1^tSA_1 y\right)$$
where 
$$ V(y)=\frac{ [y]_3 A_1^tSA_1y-  [A_1^tSA_1 y]_3y}{[A_1y]_3 [A_1^tSA_1 y]_3}.  $$
This vector satisfies:  $[V(y)]_3=0$ and $V(y_1)=0$. In particular, $T_2T_1(x,y_1)=(x,y_1)$ whenever $(x,y_1)$ is in the domain of $T_2T_1$.
\end{proposition}

 In order to study this return map in a neighborhood of $(x,y_1)$ 
 we use spherical coordinates about the axis $y_1$:
 \begin{equation}\label{y} y= \cos\psi\, y_1 + \sin\psi\cos\varphi\, \hat{\epsilon}_1 + \sin\psi\sin\varphi\, \hat{\epsilon}_2\end{equation}
 where 
$$ \hat{\epsilon}_1:=\epsilon_1,\ \  \hat{\epsilon}_2:=\frac1{\sqrt{1-\cos^2(\beta/2)\cos^2\phi}}(\sin(\beta/2)\cos\phi\, \epsilon_2+\sin\phi\, \epsilon_3),\ \  \hat{\epsilon}_3:=y_1 $$
form an orthonormal frame.     See Figure \ref{directions}.   (Notice the typographical  distinction between   the corner angle $\phi$ of the wedge domain  and the spherical coordinate $\varphi$.)
Let 
$$\left(X(x,\varphi, \psi), Y(x, \varphi,\psi)\right):= T_2T_1(x, \cos\psi\,  y_1+\sin\psi\cos\varphi \, \hat{\epsilon}_1 + \sin\psi\sin\varphi\,  \hat{\epsilon}_2) $$
and define
$$w:=w(\varphi):= \cos\varphi\, \hat{\epsilon}_1 + \sin\varphi\, \hat{\epsilon}_2.$$
Thus  we may write $y=\cos\psi\left(y_1+\tan\psi w(\varphi)\right)$. Since the rotation $S_2:=SA_1^tSA_1$ fixes $y_1$, it acts on $w$ as
$S_2w(\varphi)=w(\varphi+\theta)$ for some constant angle $\theta$. 
It follows that 
$$ S_2y= \cos\psi\, y_1 +\sin\psi\, w(\varphi+\theta).$$
     The following proposition summarizes these observations and notations.

   \begin{proposition}\label{R}
   For points $y\in  S^2_+$ in a  neighborhood of $y_1$   we adopt spherical coordinates relative to the axis $y_1=\hat{\epsilon}_3$, so that  $y=\cos\psi \left(y_1+ \tan\psi\, w(\varphi)\right)$ where 
   $$w:=w(\varphi):= \cos\varphi\, \hat{\epsilon}_1 + \sin\varphi\,\hat{\epsilon}_2.$$
   See Figure \ref{directions}.
   We also use  the notations $[z]_3:=z\cdot\epsilon_3$,
   $S_1:= A_1^{-1}SA_1$, and  $S_2=SA_1^{-1}SA_1$.
  Let $R:=T_2T_1$ be the $2$-step return map as defined above, whose domain contains a neighborhood of $(x,y_1)$ for all $x\in \mathbb{R}^2$.
  Then $R(x,y_1)=(x,y_1)$ for all $x$ and 
  $$R(x,  y_1+ \tan\psi\, w(\varphi)) =(X,y_1+ \tan\psi\, S_2w(\varphi))=(X, \tan\psi\, w(\varphi+\theta))$$ for an angle $\theta$, depending only on the wedge angle $2\phi$  and the characteristic angle $\beta$ of the no-slip reflection, such that
\begin{align*}
\cos \theta &=(S_2\hat{\epsilon}_1)\cdot \hat{\epsilon}_1= 1-8\delta^2+8\delta^4\\
\sin\theta &=(S_2\hat{\epsilon}_1)\cdot\hat{\epsilon}_2 =4\delta(1-2\delta^2)\sqrt{1-\delta^2}
\end{align*}
  where $\delta:=\cos(\beta/2)\cos\phi$.
  Writing $(X, \Phi, \Psi)=R(x, \varphi,\psi)$ we have
\begin{equation}
R: \begin{cases}
X&=x + \tan \psi\frac{[A_1(x-\gamma y_1)]_3}{[y_1]_3}
\frac{ (I+S_1)w-\frac{[(I+S_1)w]_3 y_1}{[y_1]_3} + \tan\psi \frac{[w]_3 S_1w- [S_1w]_3 w}{[y_1]_3}}{1 - \tan\psi\left(\frac{[(A_1+S_1)w]_3}{[y_1]_3} - \tan\psi\frac{[A_1w]_3[S_1w]_3}{[y_1]_3^2}\right)}\\
\Phi &=\varphi + \theta\\
\Psi &=\psi
\end{cases}
\end{equation}
Denoting $\mu_1:=\zeta_1^{-1}\epsilon_3\in \mathbb{R}^2$,  we further have $X(x+s\mu_1, \varphi,  \psi)=X(x,\varphi,\psi)+s\mu_1.$
\end{proposition}

Since $\psi$ remains constant under iterations of the return map $R=T_2T_1$, we regard $\psi$ as a fixed parameter. We are interested in 
small values of $r:=\tan\psi$.
Notice that $[A_1z]_3:=(A_1 z)\cdot\epsilon_3=z\cdot(A_1^t \epsilon_3)=\mu_0\cdot z$, where 
$$\mu_0:=A_1^t\epsilon_3=\left(\begin{array}{c}\cos(\beta/2) \sin (2\phi) \\ \sin(\beta/2)\sin(2\phi) \\-\cos(2\phi)\end{array}\right).$$
Write $x_0:=\alpha y_1$, so
$$x_0=2\sin\phi\left(\begin{array}{c}0 \\ -\sin\phi \\ \sin(\beta/2)\cos\phi\end{array}\right). $$
Then the proposition shows that $R$ has the form
\begin{equation}\label{map}
R: (x, \varphi)\mapsto \left(X= x + \mu_0\cdot (x-x_0) V_r(\varphi), 
\Phi= \varphi+\theta\right)
\end{equation} 
where the vector $V_r(\varphi)$ can be made arbitrarily (uniformly) small by choosing $\psi$ (or $r=\tan\psi$) sufficiently close to $0$.  Observe from the explicit form  $$V_r(\varphi)=
\frac{1}{[y_1]_3}\frac{ r\left((I+S_1)w-\frac{[(I+S_1)w]_3 y_1}{[y_1]_3}\right) +r^2 \frac{[w]_3 S_1w- [S_1w]_3 w}{[y_1]_3}}{1 - r\frac{[(A_1+S_1)w]_3}{[y_1]_3} + r^2\frac{[A_1w]_3[S_1w]_3}{[y_1]_3^2}}$$ 
 that $V_r(\varphi)\cdot \epsilon_3=0$
so that $X$ is indeed in  $\mathbb{R}^2$.

\begin{proposition}\label{inv}
The quantity $1+\mu_0 \cdot V_r(\varphi)$ satisfies the coboundary relation
\begin{equation}\label{coboundary}1+\mu_0\cdot V_r(\varphi)=\frac{\rho(\varphi)}{\rho(\varphi+\theta)}\end{equation}
where $$\rho(\varphi)=1+ r \frac{\tan\phi}{\sin(\beta/2)}\sin \varphi.$$
In fact, the transformation $R$ on the $3$-dimensional space $\mathbb{R}^2\times \mathbb{R}/(2\pi \mathbb{Z})$,  obtained by fixing a value of $\psi$ (hence of
$r=\tan\psi$),   leaves invariant
the measure   $$d\mu= c\left(1 + r \frac{\tan\phi}{\sin(\beta/2)}\sin \varphi\right)dA\, d\varphi$$
where $c$ is a positive constant (only dependent on the fixed parameters $\beta, \psi, \phi$) and $A$ is the standard area measure on $\mathbb{R}^2$.
\end{proposition}
\begin{proof}
The canonical invariant measure on $\mathbb{R}^2\times S^2_+$ has the form $y\cdot \epsilon_3\, dA\, dA_S$, where  $A_S$ is the area measure on $S^2_+$.
For a fixed value of $\psi$ we obtain an invariant measure on $\mathbb{R}^2\times S^1$ of the form $y\cdot \epsilon_3\, dA\, d\varphi$. Using the form of $y$ given by (\ref{y}), one   obtains  $$y\cdot \epsilon_3=\frac{\cos\psi \cos\phi\sin(\beta/2)}{\sqrt{1-\cos^2(\beta/2)\cos^2\phi}}\left(1 + r \frac{\tan\phi}{\sin(\beta/2)}\sin \varphi\right).$$
This shows that, up to a multiplicative constant, the invariant measure $\mu$ has the indicated form.   Equation (\ref{coboundary}) is an easy consequence of the invariance of $\mu$ with respect to $R$. 
\end{proof}

\begin{figure}[htbp]
\begin{center}
\includegraphics[width=1.7 in]{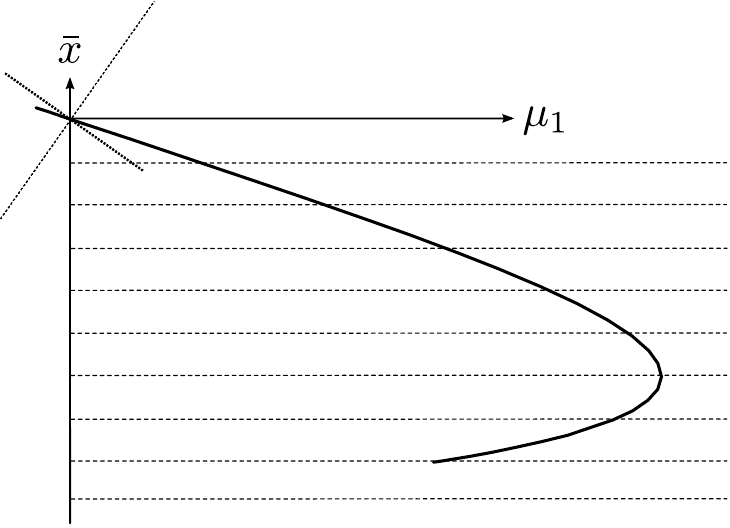}\ \ 
\caption{\small{The map $R$ sends fibers $x+ \mathbb{R}\mu_1$ onto other such fibers preserving length. That is, $dR_{(x,\varphi)}\mu_1=\mu_1$.
The quotient is a measure preserving transformation on $\mathbb{R}\times \mathbb{T}^1$. The coordinate on the first factor of the quotient is 
$\bar{x}=x\cdot \mu_0$. The curve shown above is typical  of the set to which orbits of $R$ project in $\mathbb{R}^2$. }}
\label{InvSet}
\end{center}
\end{figure}

By using the coordinate system $(\bar{x},\bar{y})\mapsto \bar{x}\mu_0+\bar{y}\mu_1$ on $\mathbb{R}^2$, the area measure is $dA=d\bar{x}\, d\bar{y}$ and, as observed at the end of Proposition \ref{R}, the transformation $R$ maps the fibers of the 
projection $(\bar{x},\bar{y})\mapsto \bar{x}$ to fibers preserving the length measure on fibers. Thus we obtain  a transformation $\bar{R}$ on $\mathbb{R}\times S^1$ preserving
the measure $ d\bar{\mu}(\bar{x},\varphi)= \rho(\varphi)\, d\bar{x}\, d\varphi$
where $\rho(\varphi)$ has the stated expression. 
Using the  quotient coordinates $\bar{x}=x\cdot\mu_0$ and $\phi$ and writing $\overline{V}_r(\varphi):=V_r(\varphi)\cdot \mu_0$ we obtain
$$\overline{R}(\bar{x}, \phi)= \left((1+ \overline{V}_r(\varphi)) \bar{x} -\bar{x}_0 \overline{V}_r(\varphi) , \phi+\theta\right). $$
In particular,
$$ \overline{X}=\frac{\rho(\varphi)}{\rho(\varphi+\theta)} \bar{x} + \left(1-\frac{\rho(\varphi)}{\rho(\varphi+\theta)}\right)\bar{x}_0.$$
The invariant measure is
$$d\bar{\mu}(\bar{x},\varphi)=\rho(\varphi) d\bar{x} d\varphi $$
where $\rho(\varphi)$ is the density given in Proposition \ref{inv}.
It is now immediate that
$$ \overline{R}^n(\bar{x}, \varphi)= \left(\frac{\rho(\varphi)}{\rho(\varphi+n\theta)}\bar{x}+\left(1-\frac{\rho(\varphi)}{\rho(\varphi+n\theta)}\right)\bar{x}_0,\varphi+n\theta\right).$$ 
This shows that all the iterates of $(\bar{x},\varphi)$ remain uniformly close to the initial point for small values of $\psi$.  Also notice
that $(\zeta_1\mu_0)\cdot e_{2,1}= \nu_2\cdot e_{2,1}=\sin(2\phi)>0$. This means that if $\bar{x}$ remains bounded, the length coordinate
along the base of $\mathcal{P}_1$ also must be similarly bounded.  From this we conclude:

\begin{corollary}\label{local_stability} Assume the notation introduced at the beginning of this section.
For all $q\in \mathcal{P}_i\setminus (\mathcal{P}_1\cap \mathcal{P}_2)$, $i=1,2$,   and any  neighborhood $\mathcal{V}$ of the period-$2$ state  $(q, v_i)\in S^+_i$, there
exists  a small enough neighborhood $\mathcal{U}\subset \mathcal{V}$  of $(q, v_i)$  the orbits of whose points remain in $ \mathcal{V}$.
\end{corollary}

Because any (bounded) polygonal billiard shape must have a corner with angle less than $\pi$, the following corollary holds.

\begin{theorem} Polygonal no-slip billiards cannot be ergodic for the canonical invariant measure.\end{theorem}

\section{Higher order periodic orbits in polygons}\label{higher order polygon}
The analysis of the previous section is based on the existence of period-$2$ orbits in wedge-shaped no-slip billiard tables. 
Existence of periodic orbits of higher periods is in general difficult to establish, although one such result for wedge domains will be indicated below in this section.
We first point out a generalization of Corollary \ref{local_stability} to perturbations of periodic orbits in general polygon-shaped  domains.

\begin{figure}[htbp]
\begin{center}
\includegraphics[width=4.5 in]{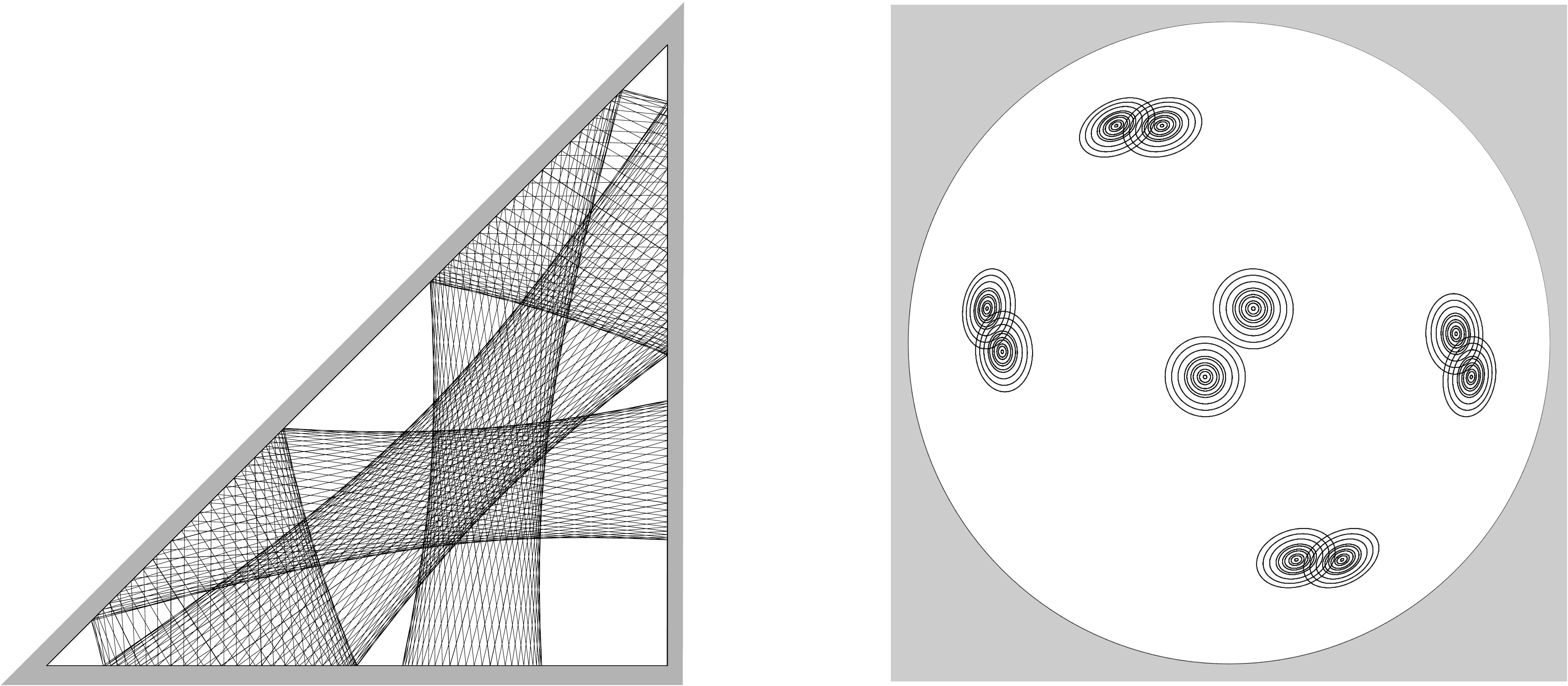}\ \ 
\caption{\small{Projection to the plane  of orbits in the neighborhood of a period-$10$ periodic orbit of a triangular no-slip billiard domain (left) showing typical stable behavior, along with the velocity phase portrait projections (right). }}
\label{near_periodic}
\end{center}
\end{figure} 

Figure \ref{near_periodic} illustrates the type of stability implied by the following Theorem \ref{polygon stability}.
\begin{theorem}\label{polygon stability}
Periodic orbits in no-slip polygon-shaped billiard domains are locally stable. That is, given an initial state $\xi_0=(q_0,v_0)$ for a  period-$n$ orbit in such a billiard system, and
for any neighborhood $\mathcal{V}$ of $\xi_0$, there exists  a small enough neighborhood $\mathcal{U}\subset \mathcal{V}$ of $\xi$ the orbits of whose elements  remain in $\mathcal{V}$. 
\end{theorem}
\begin{proof}
The idea is essentially the same as used in the proof of Proposition \ref{inv} and Corollary \ref{local_stability}. We only indicate the outline. By a choice of convenient coordinates around the periodic point, it is possible to show that the $n$-th iterate of the billiard map $T$, denoted $R:=T^n$, can be regarded as a map from an open subset of $\mathbb{R}^2\times S^1$ into this latter set, having the form $R(x,\varphi)= (x_0+ A(\varphi)(x-x_0), \varphi+\theta)$ for a certain angle $\theta$, where $A(\varphi)$ is a linear transformation independent of $x$.  Rotation invariance implies that $R$   must satisfy the invariance property $R(x+su,\varphi)=R(x,\varphi)+s u$ for a   vector $u\in \mathbb{R}^2$. From this we define a map $\overline{R}$ on (a subset of) the quotient $\mathbb{R}\times S^1$, $\mathbb{R}^2/\mathbb{R}u$. Furthermore, denoting by $(\bar{x}, \varphi)$ the coordinates in this quotient space,    invariance of the canonical measure implies invariance of a measure $\mu$ on this quotient having the form $d\mu(\bar{x},\varphi)=\rho(\varphi)\, d\bar{x}\, d\varphi.$ Invariance is with respect to the quotient map $\overline{R}(\bar{x},\varphi)=(\overline{x}_0 + a(\varphi)(\bar{x}-\bar{x}_0), \varphi+\theta)$ for some function $a(\varphi)$. This function must  then take the form $a(\varphi)= \rho(\varphi)/\rho(\varphi+\theta)$. Iterates of $\overline{R}$ will then
behave like the corresponding map  for the wedge domain, defined prior to Theorem \ref{local_stability}.
\end{proof}

\begin{figure}[htbp]
\begin{center}
\includegraphics[width=3.0 in]{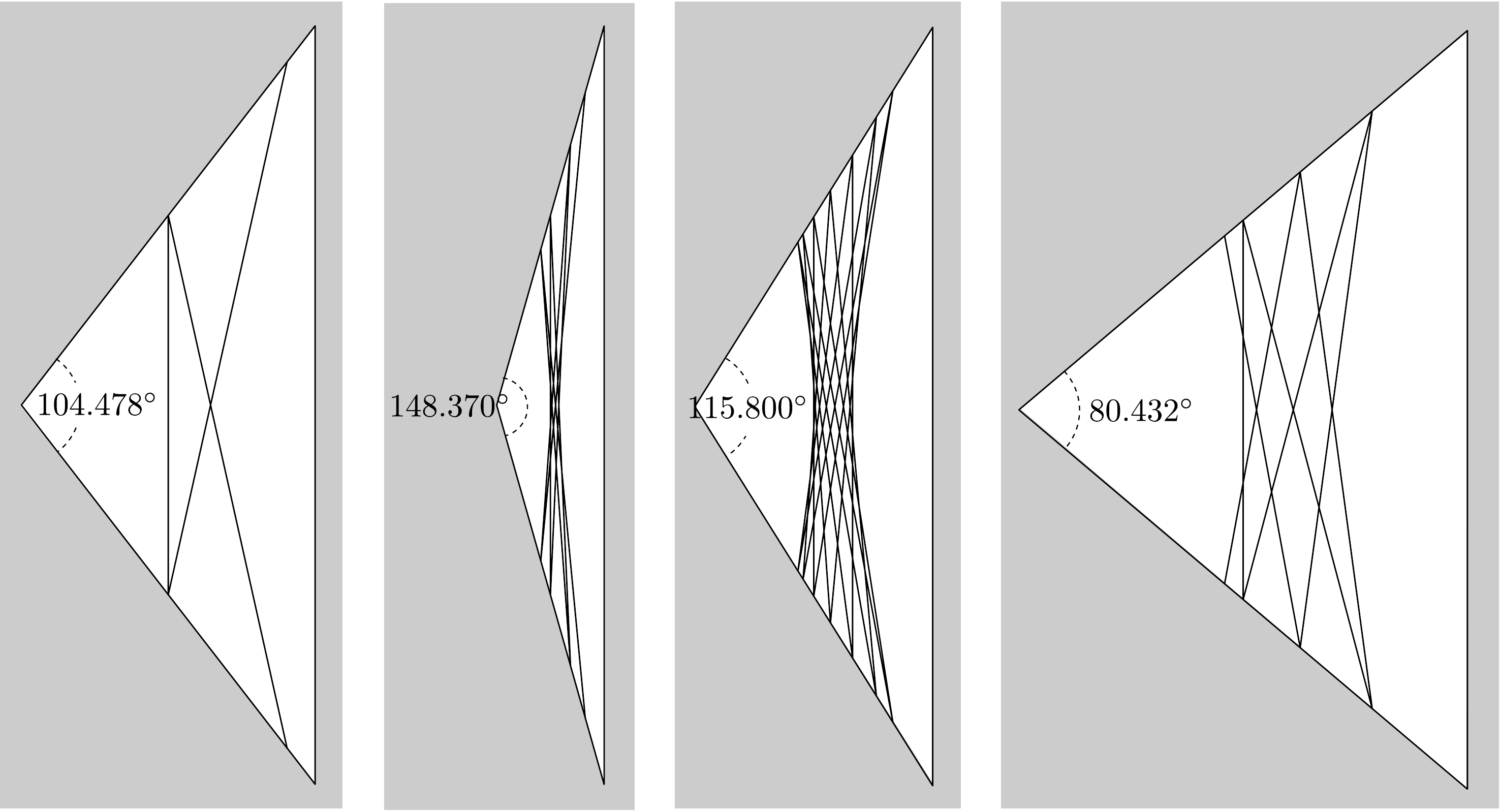}\ \ 
\caption{\small{From left to right:  projections to the plane of periodic orbits of periods $6, 14, 14, 14.$ (Bounded orbits in the same wedge domain are all periodic with the same period.)
The rotation angle $\theta$ in each case  is $2\pi p/q$ where $p/q=1/3, 1/7, 2/7,3/7$. Mass distribution  of the disc particle is uniform.}}
\label{periodpq1}
\end{center}
\end{figure}

We turn now to the question of existence of periodic orbits of higher (necessarily even) periods for wedge shapes. Clearly, a necessary condition is that the angle $\theta$ introduced in Proposition \ref{R} (see also Figure \ref{directions}) be rational. 
For orbits that do not eventually escape to infinity, this is also a sufficient condition, as a simple application of Poincar\'e recurrence shows. (See \cite{CFII}.)
Moreover, as $\theta$ is only a function of $\delta:=\cos(\beta/2)\cos\phi$, which is given by (Proposition \ref{R})
\begin{equation}\label{delta}
\cos \theta= 1-8\delta^2+8\delta^4
\end{equation}
where $\beta$ is the characteristic angle of the system (a function of the mass distribution on the disc) and $2\phi$ is the corner angle of the wedge domain, if a higher order periodic orbit exists for a given $\delta$, all bounded orbits have the same period.

\begin{figure}[htbp]
\begin{center}
\includegraphics[width=4.5 in]{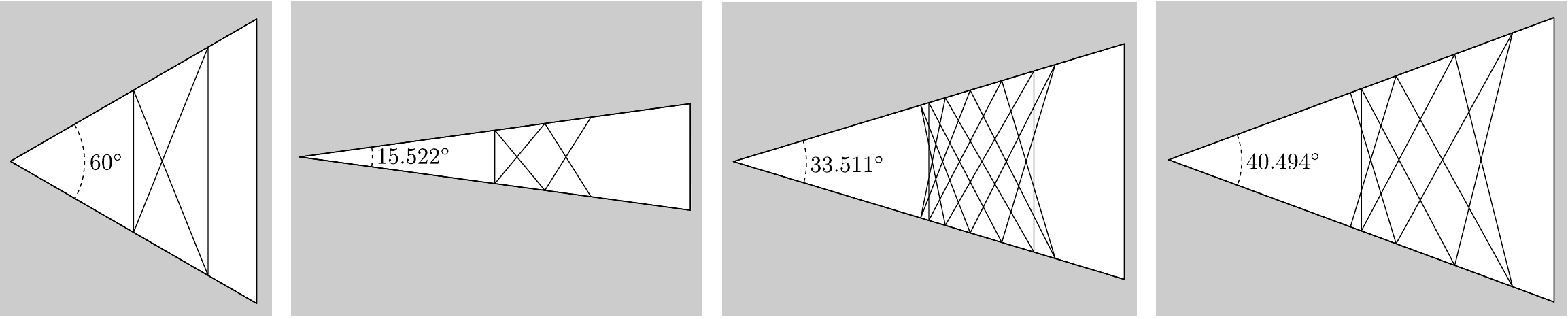}\ \ 
\caption{\small{From left to right: projections of periodic orbits of periods $4, 10, 14, 18$.  The rotation angle $\theta$ in each case is $2\pi p/q$ where $p/q$ is $1/2, 2/5, 3/7, 4/9$, respectively. Mass distribution is uniform.}}
\label{periodpq}
\end{center}
\end{figure} 

We give  a few examples for the uniform mass distribution, for which $\cos(\beta/2)=\sqrt{2/3}.$
Solving \ref{delta} for $\cos\phi$, for $\theta=2\pi p/q$, choosing first  the negative square root, gives
\begin{equation}\label{phipq} \cos \phi_{p,q}:=\frac{\sqrt{3}}{2} \sqrt{1- \sqrt{\frac{1+\cos({2\pi p}/{q})}{2}}}\end{equation}
A few examples are shown in Figure \ref{periodpq1}.

Notice that there are no restrictions on the values of $p$ and $q$. 
The following proposition is a consequence of these remarks.
\begin{theorem} For any positive even integer $n$ there exists a wedge domain for which the no-slip billiard has period-$n$ orbits. More specifically,
all bounded orbits of the no-slip billiard in a wedge domain with corner angle 
$\phi_{p,q}$ satisfying Equation \ref{phipq}
 are periodic with period $2q$.
\end{theorem}

 \begin{figure}[htbp]
\begin{center}
\includegraphics[width=3.5 in]{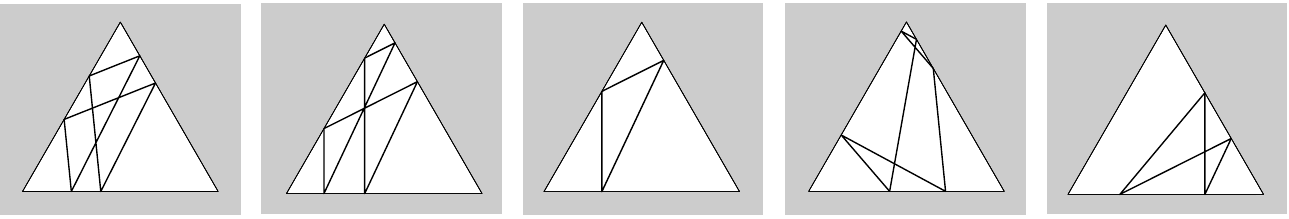}\ \ 
\caption{\small{All orbits of an equilateral triangle  no-slip billiard system are periodic with (not necessarily least) period equal to $4$ or $6$. }}
\label{equilateral}
\end{center}
\end{figure}

Solving \ref{delta} for $\cos\phi$, for $\theta=2\pi p/q$, but choosing now the positive square root, gives
$$ \cos \phi_{p,q}:=\frac{\sqrt{3}}{2} \sqrt{1+ \sqrt{\frac{1+\cos({2\pi p}/{q})}{2}}}.$$
This makes sense so long as $0.392\approx \arccos(-7/9)/2\pi \leq p/q\leq 0.5$, which greatly restricts the choices of $p$ and $q$.  A few examples in this case are shown in Figure \ref{periodpq}.

  It is interesting to observe that all orbits of the equilateral triangle  are periodic with period  $4$ or $6$. (See Figure \ref{equilateral} and \cite{CFII} for the proof.) We do not know of any other no-slip billard domain all of whose orbits are periodic.
 
  \begin{figure}[htbp]
\begin{center}
\includegraphics[width=3.0 in]{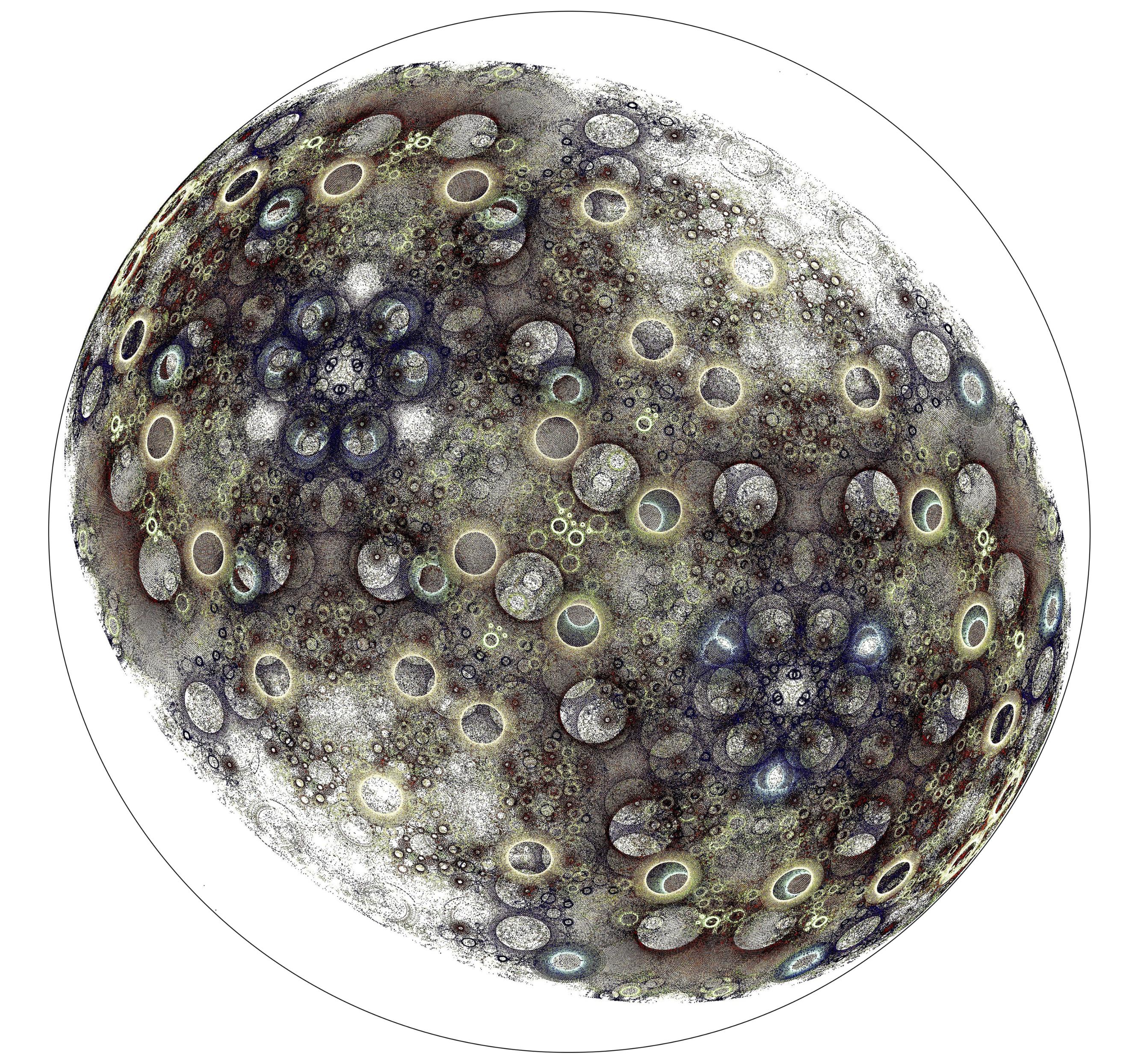}\ \ 
\caption{\small{Velocity phase portrait of the no-slip billiard on a regular pentagon. Orbits all seem to lie in a stable neighborhood of some periodic orbit. }}
\label{moon}
\end{center}
\end{figure}

A final observation concerning polygonal no-slip billiards is suggested by  plots of their velocity phase portraits. A typical such plot is shown  
in Figure \ref{moon}.  It is apparent that the orbits drawn all seem to lie on a stable neighborhood of some periodic  orbit, and this pattern is seen at
all scales that we have explored, but we do not yet have  a clear topological dynamical description of this observation.

\section{Linear stability in the presence of curvature}\label{curved section}
We now turn to the problem of characterizing stability of period-$2$ orbits for no-slip billiard domains whose boundary may have non-zero geodesic curvature. 
Here we only address {\em linear} rather than local stability as we did before for polygonal billiards.  In other words, we limit ourselves to the problem of determining when
the differential of the billiard map $dT_\xi$ at a period-$2$ collision state $\xi=(q,v)$ is elliptic or hyperbolic, and precise thresholds (where it is parabolic). 
To go from this information to local stability would require developing a KAM theory for no-slip billiards in the model of \cite{sevryuk}, something we do not do in this paper.

A simple but key observation is contained in the following lemma.
\begin{lemma}
Let $\xi=(q,v)$ be periodic with  period $2$ for the no-slip billiard map and consider the differential $\mathcal{T}:=dT^2_\xi:v^\perp\oplus v^\perp\rightarrow v^\perp\oplus v^\perp$. Then either all the eigenvalues of $\mathcal{T}$ are real, of the form $1,1, r, 1/r$ or, if not all real, they are $1, 1, \lambda, \overline{\lambda}$ where $|\lambda|=1$. 
\end{lemma}
\begin{proof}
This is a consequence of the following observations. First, we know that $T^*\Omega=-\Omega$, where $\Omega$ is the canonical symplectic form (cf. Section
\ref{measure reversible}). Therefore, the product of the eigenvalues of $\mathcal{T}$ counted with multiplicity is $1$. 
The vector $(e_1, w_1)$, where $e_1$ is the first vector in  the product frame and $w_1$ is the first vector in  the wavefront frame,
is an eigenvector for eigenvalue $1$ of $dT_\xi$ due to rotation symmetry, as already noted. If we regard $dT_\xi$ as a self-map of $v^\perp\oplus v^\perp$ as in the corollary to 
Proposition \ref{differential} then we should use instead the vector $(w_1, w_1)$. (Recall that $w_1$ is collinear with the orthogonal projection of $e_1$ to $v^\perp$.)
 In addition, by reversibility of $T$, if $\lambda$ is an eigenvalue of $\mathcal{T}$, then $1/\lambda$ is one also, and since $\mathcal{T}$ is a real valued linear map, the complex conjugates
 $\overline{\lambda}$ and $1/\overline{\lambda}$ are also eigenvalues. As the dimension of the linear space is $4$, if one of the eigenvalues, $\lambda$,  is not real, it must be the case that $\lambda=1/\overline{\lambda}$ and we are reduced to the case $1, 1, \lambda, \overline{\lambda}$ with $\lambda\overline{\lambda}=1$. 
If all eigenvalues are real,
and $r\neq 1$ is one eigenvalue, then we are reduced to the case $1, 1, r, 1/r$.  
\end{proof}

  \begin{corollary}\label{trace corollary}
  The period-$2$ point $\xi$ is elliptic for $\mathcal{T}=dT^2_\xi$ if and only if $|\text{\em Tr}(\mathcal{T})-2|<2.$
  \end{corollary}

 To proceed, it is useful to express the differential map  of Corollary \ref{cor} in somewhat different form. First observe, in the period-$2$ case (in which $\tilde{v}=-v$ and $v^\perp={\tilde{v}}^\perp$), that
  $$w_2(\xi)=-w_2(\tilde{\xi})\ \text{and} \  \frac{\cos\psi(\tilde{q},v)}{\cos \phi(\tilde{q},v)}=\frac{\cos\psi(\tilde{\xi})}{\cos \phi(\tilde{\xi})}=\frac{\cos\psi({\xi})}{\cos \phi({\xi})}.$$
  (See Section \ref{periodic section}.)
  Now define the rank-$1$ operator
  $$\Theta_{\tilde{\xi}}(Z):= 2 \cos(\beta/2)\frac{\cos\psi(\tilde{\xi})}{\cos\phi(\tilde{\xi})} \left\langle w_2(\tilde{\xi}), Z\right\rangle u_1(\tilde{q}). $$
 Then
\begin{equation}\label{diffT} dT_{\xi}\left(\begin{array}{c}X \\Y\end{array}\right)=  \left(\begin{array}{cc}I & tI \\-\kappa(\tilde{q})\Theta_{\tilde{\xi}} & C_{\tilde{q}} - t\kappa(\tilde{q})\Theta_{\tilde{\xi}}\end{array}\right)       \left(\begin{array}{c}X \\Y\end{array}\right). \end{equation}

When the geodesic curvature satisfies $\kappa(q)=\kappa(\tilde{q})$ we obtain a simplification in the criterion for ellipticity, as will be seen shortly. With this special case in mind we 
define the linear map  $R_\xi $ on $v^\perp$ by $R_\xi w_i(\xi)=-(-1)^iw_i(\xi)$, $i=1,2.$ Notice that $Ru_1(q)=u_1(\tilde{q})$. Then
 $$ R_{\tilde{\xi}}  C_{\tilde{q}} = C_q R_\xi, \ \ R_{\tilde{x}}\Theta_{\tilde{\xi}}= \Theta_\xi\, R_\xi.$$
The same notation $R_\xi$ will be used   for the map on $v^\perp\oplus v^\perp$ given by $(z_1,z_2)\mapsto (R_\xi z_1, R_\xi z_2)$. Notice that $R:=R_\xi=R_{\tilde{\xi}}$
since $w_i(\tilde{\xi})= -(-1)^i w_i(\xi)$. It follows that
\begin{equation}\label{diffTbiz} R dT_{\xi} R=  \left(\begin{array}{cc}I & tI \\-\kappa(\tilde{q})\Theta_{{\xi}} & C_{{q}} - t\kappa(\tilde{q})\Theta_{{\xi}}\end{array}\right). \end{equation}
In particular, when $\kappa(q)=\kappa(\tilde{q})$, we have
$ R dT_\xi R = dT_{\tilde{\xi}}$ and $dT^2_\xi = (R dT_\xi)^2$. 
Therefore, rather than computing the trace of $dT^2_\xi$, we need only consider the easier to compute trace of $RdT_\xi$. A straightforward calculation  gives the trace of these maps, which we record  in the next lemma.

\begin{lemma}\label{traces}
Let $\xi=(q,v)$ have period $2$ and set  $\tilde{\xi}:= T(\xi)$,  ${C}:=C_{q}, {\Theta}:= \Theta_q.$
Then
$$\text{\em Tr}\left(dT^2_\xi\right) =\text{\em Tr}\left\{I+ (C R)^2 - t(\kappa(q)+\kappa(\tilde{q}))\left[\Theta + (C R) (\Theta R)\right] + t^2 \kappa(q)\kappa(\tilde{q}) (\Theta R)^2\right\}.  $$
When $\kappa:=\kappa(\tilde{q})=\kappa(q)$, we have
$ \text{\em Tr}(R dT_\xi)=\text{\em Tr}\left(CR + t\kappa\Theta\right).$
\end{lemma}
\begin{proof}
These expressions follow easily given the above definitions and notations.
\end{proof}
These traces can now be computed using Equations (\ref{Cwavefront}) and (\ref{inner}). The matrices expressing  $C, R, \Theta$ in the wavefront basis of $v^\perp$ are given as follows. For convenience we write  
  $$c:=\cos(\beta/2),\  c_\phi:=\cos\phi,\  c_\psi:=\cos\psi,\  \varrho:= \sqrt{1-\cos^2(\beta/2)\cos^2\phi},$$ where $\phi=\phi(\xi)$ and $\psi=\psi(\xi)$ are defined in Corollary \ref{cor}.
$$[C]_w=\left(\begin{array}{cc}1-2c^2 c^2_\phi & -2 c c_\phi \varrho \\-2cc_\phi\varrho & -1+2 c^2 c^2_\phi\end{array}\right), \ \ [R]=\left(\begin{array}{cr}1 & 0 \\0 & -1\end{array}\right), \ \ [\Theta]_w=2 c \frac{c_\psi}{c_\phi}\left(\begin{array}{cc}0 & \varrho \\0 & -cc_\phi\end{array}\right). $$

Let $\bar{d}$ be the distance between the projections of $q$ and $\tilde{q}$ on plane the billiard table,  $\overline{v}$ the projection of $v$ on the same plane and $t$, as before,  the time between consecutive collisions. From $\cos\psi = \sin(\beta/2) \cos\phi /\sqrt{1-\cos^2(\beta)\cos^2\phi}$ it follows that
$t\cos\psi = \cos\phi\,  \bar{d}$.  We then obtain 
\begin{equation}\label{equal k} \text{Tr}\left(RdT_\xi\right)=\text{Tr}(CR) + t\kappa \text{Tr}(\Theta)= 2\left[1- 2\cos^2(\beta/2)\cos^2\!\phi\right] -2\kappa \bar{d}  \cos^2(\beta/2)\cos\phi\end{equation}
and 
\begin{align}\label{general trace}
\begin{split}
\text{Tr}\left(dT_\xi^2\right)        &=4\left\{\left[1-2\cos^2(\beta/2)\cos^2\!\phi\right]^2\right.\\
                                                  &\ \ \ \ \ \ \ \ \ \ \ \ \ \ \ -(\kappa(q)+\kappa(\tilde{q}))\cos^2(\beta/2)\cos\phi \left[1-2\cos^2(\beta/2)\cos^2\!\phi\right]\bar{d}\\
                                                  &\ \ \ \ \ \ \ \ \ \ \ \ \ \ \ \ \ \ \ \ \ \ \ \ \ \ \ \ \ \ \left.+ \kappa(q)\kappa(\tilde{q}) \cos^4(\beta/2)\cos^2\! \phi\, {\bar{d}}^2\right\}
\end{split}
\end{align}
Observe that in the special case in which $\kappa(q)=\kappa(\tilde{q})$ we have
$$ \text{Tr}\left(dT_\xi^2\right)=\left\{2\left[1-2\cos^2(\beta/2)\cos^2\phi\right]- 2\kappa \cos^2(\beta/2)\cos\phi\, \bar{d}\right\}^2$$

\begin{theorem}\label{general corner}
Suppose that the billiard domain has a piecewise smooth boundary with at least one corner  having inner angle less than $\pi$. Then, arbitrarily close to that
corner point,  the no-slip billiard   has (linearly) elliptic period-$2$ orbits. 
\end{theorem}
\begin{proof}
Period-$2$ orbits exist arbitrarily close to the corners of a piecewise smooth billiard domain as Figure \ref{corner} makes clear.
 For period-$2$ orbits near a corner the above expression for $\text{Tr}\left(dT^2_\xi\right)$ gives for small $\bar{d}$
$$ 0<\text{Tr}\left(dT^2_\xi\right) = 4\left[1-2\cos^2(\beta/2)\cos^2\phi\right]^2 +O(\bar{d})<4.$$
This implies that 
$$ \left| \text{Tr}\left(dT^2_\xi\right)-2 \right| <2$$
and the theorem follows from Corollary \ref{trace corollary}.
\end{proof}
Theorem \ref{general corner} (and numerical experiments) strongly suggests that such no-slip billiards will aways admit small invariant open sets and thus cannot be ergodic with respect to the canonical billiard measure.

\begin{figure}[htbp]
\begin{center}
\includegraphics[width=1.3 in]{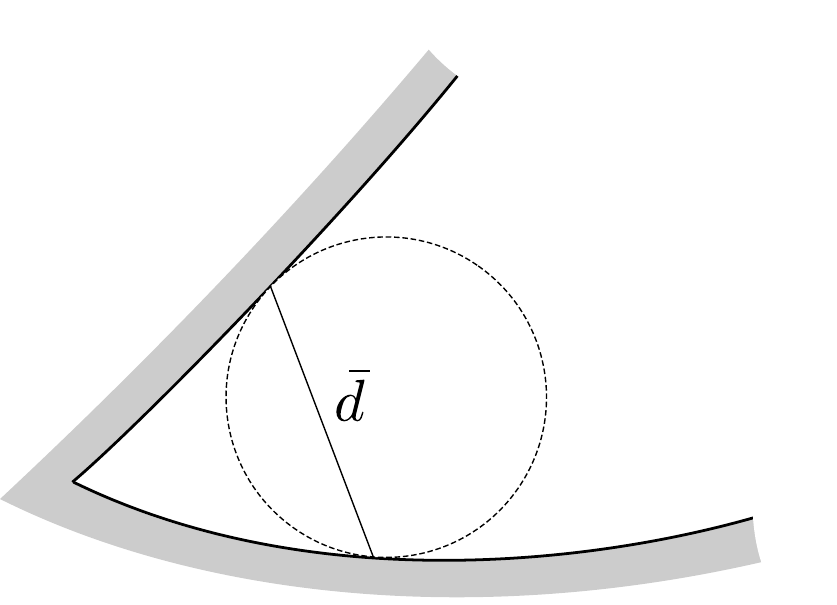}\ \ 
\caption{\small{For a  billiard domain with piecewise smooth boundary, arbitrarily near any corner  with inner angle less than $\pi$, there are linearly stable period-$2$ orbits.}}
\label{corner}
\end{center}
\end{figure}

We illustrate numerically the transition between elliptic and hyperbolic in the special case of equal curvatures at $q$ and $\tilde{q}$.
  Define $\zeta:=\kappa \bar{d}$. When $\zeta>0$ (equivalently, the curvature is positive), the critical value of $\zeta$ is 
$$ \zeta_0= \frac{2-2\cos^2(\beta/2)\cos^2\phi}{\cos^2(\beta/2)\cos\phi}.$$ The condition for the periodic point to be elliptic is $\zeta>\zeta_0$.
When $\zeta<0$, the critical value of $\zeta$ is
$$ \zeta_0= -2\cos\phi$$ and  the condition for ellipticity is  $|\zeta|<|\zeta_0|$.

 \begin{figure}[htbp]
\begin{center}
\includegraphics[width=4 in]{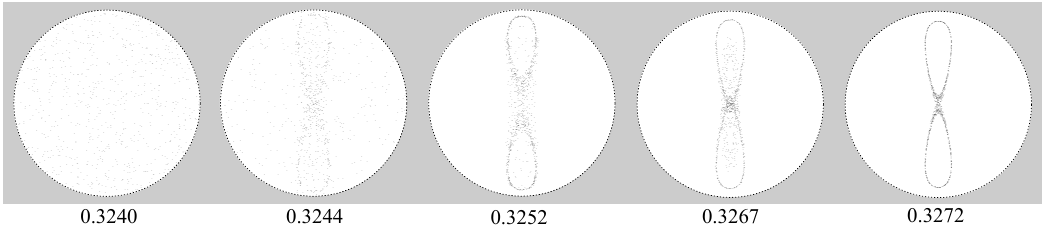}\ \ 
\caption{\small{Velocity phase portraits of single orbits near the periodic orbit of the no-slip Sinai billiard corresponding to $\phi=0$.  The mass distribution is uniform. The numbers are
the radius of the circular scatterer. }}
\label{cutoff}
\end{center}
\end{figure} 

Consider  the example of the  no-slip Sinai billiard. (See Figures \ref{example} and \ref{sinai}.)  We examine small perturbations of the periodic orbit corresponding to the angle $\phi=0$. Figure \ref{cutoff} suggests a transition from chaotic to more regular type of behavior for a radius between $0.32$ and $0.33$. In reality the critical radius for the $\phi=0$ periodic orbits is exactly $1/3$. So the observed numbers are smaller. We should bear in mind that the periodic points are not isolated, but are part of a family parametrized by $\phi$. As $\phi$ increases, the critical parameter $\zeta_0$  changes (for the uniform mass distribution, where $\cos^2(\beta/2)=2/3$) according to the expression
$\zeta_0 = (3-\cos^2\phi)/\cos\phi$. Given in terms of the radius of curvature,  $\zeta= (1-2R\cos\phi)/R$. Solving for the critical $R$ yields
$ R_0=\frac{\cos\phi}3$.  Thus for a period-$2$ trajectory having  a small but non-zero $\phi$, the critical radius is less than $1/3$. 
It is then to be expected that the experimental critical value of $R$, for orbits closed to that having   $\phi=0$  will give numbers close to but less than $1/3$.
Moreover, as   $R_0$ approaches $0$ when $\phi$ approaches $\pi/2$, we obtain the following proposition which, together with experimental evidence indicates that
the no-slip Sinai billiard is not ergodic.

\begin{figure}[htbp]
\begin{center}
\includegraphics[width=4.5in]{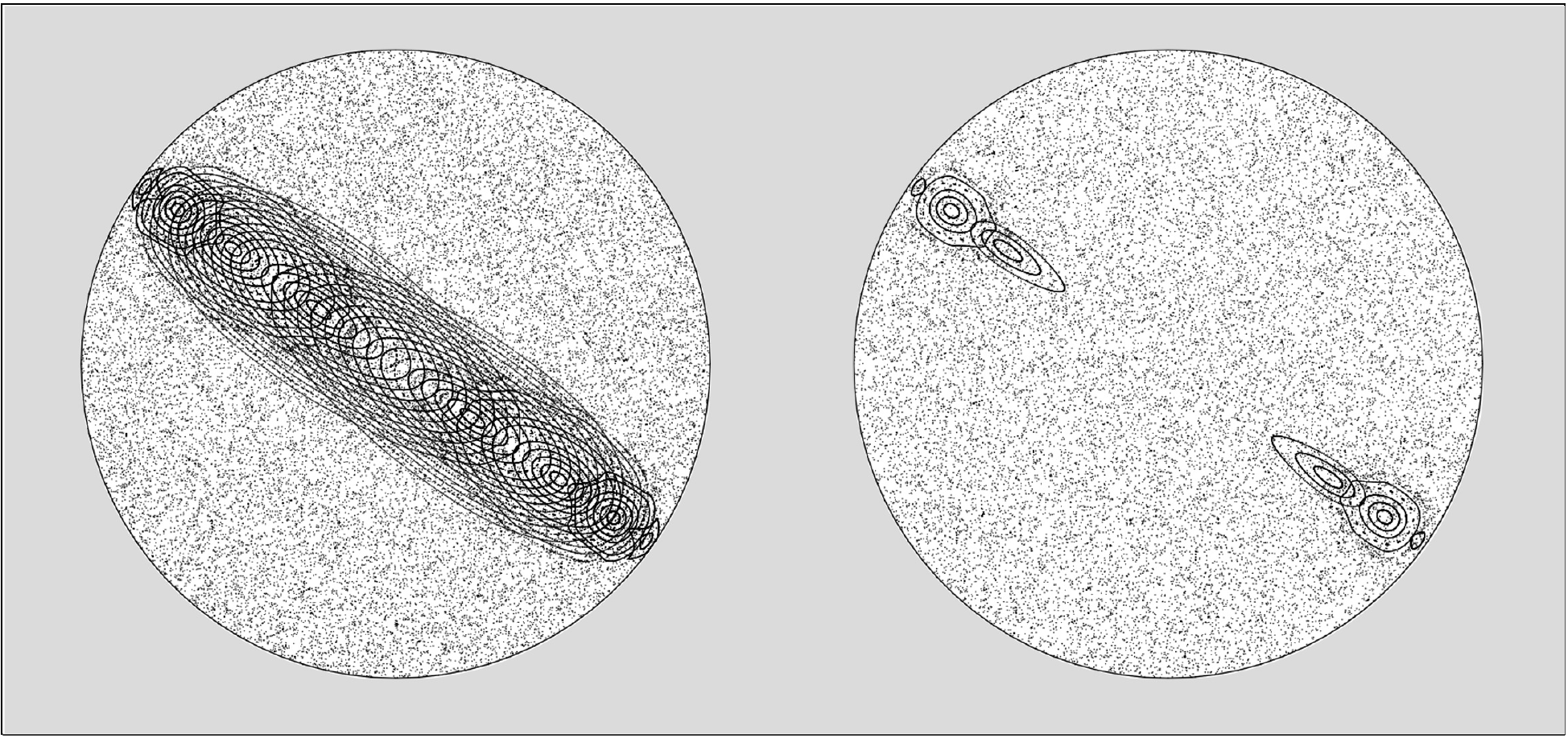}
\caption{\small{On the left: velocity phase portrait of the no-slip Sinai billiard with scatterer radius $R=0.35$. Since this is greater than the transition value $R=1/3$, the    period-$2$ orbits parametrized by $\phi$  are all  elliptic. On the right, 
  $R=0.32$ and
ellipticity has been destroyed for orbits with smaller values of $\phi$. No matter how small $R$ is, 
 elliptic orbits always exist for $\phi$ sufficiently close to $\pi/2$.}}
\label{sinaiportrait}
\end{center}
\end{figure}

\begin{proposition}
The no-slip Sinai billiard, for any choice of scatterer curvature, will contain (linearly) elliptic periodic trajectories of  period $2$.
\end{proposition}

As another example, consider the family of billiard regions bounded by two symmetric arcs of circle depicted in Figure
\ref{contact}.

 \begin{figure}[htbp]
\begin{center}
\includegraphics[width=2.5 in]{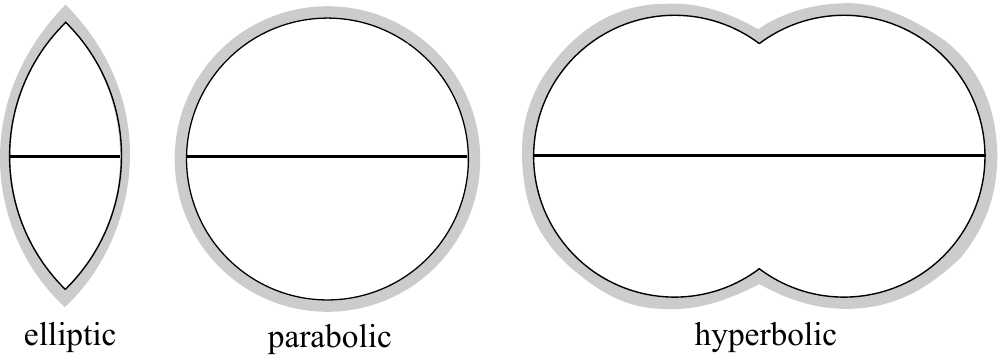}\ \ 
\caption{\small{Family of focusing no-slip billiards. }}
\label{contact}
\end{center}
\end{figure}

 In this case,  the critical transition from hyperbolic to elliptic, for the horizontal periodic orbit at middle height shown in the figure, happens for the disc. 
A   transition behavior similar to that observed for the Sinai billiard seems to occur  near the period-$2$ orbits shown  in Figure \ref{contact}.
The number indicated below each velocity phase portrait is the angle of each  circular arc. Thus, for example, the disc corresponds to angle $\pi$; smaller angles give
shapes like that on the left in Figure \ref{contact}. The cut-off angle at which the indicated periodic orbit becomes  elliptic is $\pi$. 
Notice, however, that  the experimental  value for this angle  is greater than  $\pi$. Just as in the Sinai billiard example, we should keep in mind that the periodic orbits are  not isolated; in this case, the bias would be towards greater values of the angle.

 \begin{figure}[htbp]
\begin{center}
\includegraphics[width=4 in]{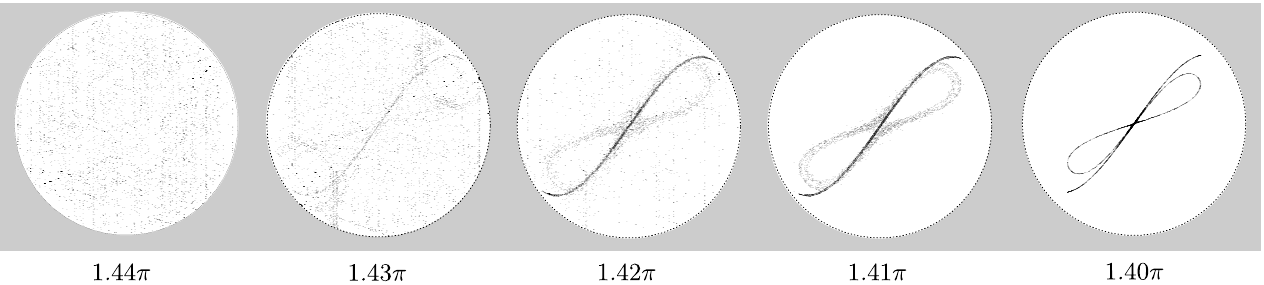}\ \ 
\caption{\small{The indicated number is the angle of the circular cap. Each depicted orbit is a perturbation of the horizontal period-$2$ orbit at the middle height, as shown in Figure \ref{contact}. Apparent chaotic behavior occurs for an angle much greater than $\pi$. This is expected since the parallel period-$2$ orbits remain linearly stable as the angle cap increases, up to a point.}}
\label{focus transition}
\end{center}
\end{figure}

\end{document}